\documentclass[11pt]{amsart}
\usepackage{fancyhdr}
\usepackage[english]{babel}
\usepackage{amssymb}
\usepackage{bbold}
\usepackage{mathrsfs}
\usepackage{enumerate}
\usepackage{xcolor}
\usepackage{ifpdf}
\usepackage[all]{xy}
\usepackage{tikz-cd}
\usepackage[initials]{amsrefs}

\usepackage{color}

\newcommand{\bbN}{{\mathbf N}}
\newcommand{\bbR}{{\mathbf R}}
\newcommand{\bbP}{{\mathbb P}}

\newcommand{\bbZ}{{\mathbf Z}}
\newcommand{\bbC}{{\mathbf C}}

\newcommand{\bbH}{\mathbf{H}}

\newcommand{\calB}{\mathcal{B}}
\newcommand{\calC}{\mathcal{C}}
\newcommand{\calD}{\mathcal{D}}

\newcommand{\calH}{\mathcal{H}}

\newcommand{\calN}{\mathcal{N}}

\newcommand{\calW}{\mathcal{W}}

\newcommand{\iv}{^{-1}}

\newcommand{\pr}{\textbf{pr}}

\newcommand{\supp}{\textbf{supp}}

\newcommand{\Isom}{\textbf{Isom}}

\newcommand{\Ad}{\textbf{Ad}}

\newcommand{\Prob}{\textbf{Prob}}
\newcommand{\half}{\frac{1}{2}}

\newcommand{\SL}{\textbf{SL}}
\newcommand{\PSL}{\textbf{PSL}}

\newcommand{\SO}{\textbf{SO}}

\newcommand{\SU}{\textbf{SU}}

\newcommand{\Hom}{\textbf{Hom}}

\newcommand{\dd }{\,{\textrm{d}}}

\newcommand{\acts}{\curvearrowright}
\newcommand{\overto}[1]{{\buildrel{#1}\over\longrightarrow}}

\newcommand{\setdef}[2]{ \left\{ \left. {#1}\ \right/\ {#2} \right\} }
\newcommand{\ip}[2]{ \left\langle {#1},{#2} \right\rangle }

\newcommand{\srad}[1]{\left|{#1}\right|_{\textbf{sp}}}

\newcommand{\NRN}{\mathcal{N}_{\textbf{RN}}}
\newcommand{\FP}{{\mathscr{FP}}}

\newcommand{\Bnd}{\textbf{Bnd}}

\newcommand{\HomZd}{\textbf{Hom}^{\textbf{Zd}}}
\newcommand{\HomZdac}{\textbf{Hom}^{\textbf{Zd}}_{\textbf{ac}}}
\newcommand{\HomZdsg}{\textbf{Hom}^{\textbf{Zd}}_{\textbf{sg}}}
\newcommand{\Csg}{{\textbf{C}_{\textbf{sg}}}}
\newcommand{\Cent}{{\textbf{c}_{\textbf{ent}}}}
\newcommand{\hRW}{{\textbf{h}_{\textbf{RW}}}}
\newcommand{\HHH}{{\mathbb{H}_1}}

\newtheorem{mthm}{Theorem}

\newtheorem{theorem}{Theorem}[section]
\newtheorem{lemma}[theorem]{Lemma}

\newtheorem*{claim}{Claim}

\newtheorem{corollary}[theorem]{Corollary}

\newtheorem{prop}[theorem]{Proposition}

\newtheorem*{problem}{Problem}

\theoremstyle{definition}

\newtheorem{notation}[theorem]{Notation}

\newtheorem{example}[theorem]{Example}

\newtheorem{remark}[theorem]{Remark}

\numberwithin{equation}{section}
\begin{document}

\title[Quotients of Poisson boundaries]{Quotients of Poisson boundaries, \\
entropy, and spectral gap}

\author{Samuel Dodds}
\address{University of Illinois at Chicago}
\email{sdodds3@uic.edu}

\author{Alex Furman}
\address{University of Illinois at Chicago}
\email{furman@uic.edu}

%\date{\today}

\begin{abstract}
Poisson boundary is a measurable $\Gamma$-space canonically 
associated with a group $\Gamma$ and a probability measure $\mu$ on it.
The collection of all measurable $\Gamma$-equivariant quotients,
known as $\mu$-boundaries, of the Poisson boundary forms a partially ordered set, 
equipped with a strictly monotonic non-negative function, known as Furstenberg 
or differential entropy. 

In this paper we demonstrate the richness and the complexity of this lattice
of quotients for the case of free groups and surface groups and rather general measures.
In particular, we show that there are continuum many
unrelated $\mu$-boundaries at each, sufficiently low, entropy level, and there are 
continuum many distinct order-theoretic cubes of $\mu$-boundaries.

These $\mu$-boundaries are constructed from dense linear representations
$\rho:\Gamma\to G$ to semi-simple Lie groups, like $\PSL_2(\bbC)^d$ with
absolutely continuous stationary measures on $\hat\bbC^d$.

\end{abstract}

\maketitle

%%%%%%%%%%%%%%%%%%%%%%%%%%%%%%%%%%%%%%%%%%%%%%%%%%%%%%%%%%%%%%%%%%%%%%%%%%%%%%%%%%%%%%%%%%%%%%%%%%%%%%%%%%%%

\section{Introduction and statement of the main results} \label{sec:intro}

Let $\Gamma$ be a countable group and $\mu$ a probability measure on $\Gamma$ such that 
$\ \supp(\mu)\ $ generates $\Gamma$ as a semi-group.
Associated with the pair $(\Gamma,\mu)$ there is a \textit{Poisson boundary} (or Furstenberg--Poisson boundary) 
$(B,\nu)$ which is a standard probability space with a measurable $\Gamma$-action for which
the probability measure $\nu$ is $\mu$-\textit{stationary}, in the sense that 
$\nu=\mu*\nu$ where $\mu*\nu=\sum \mu(\gamma)\cdot \gamma_*\nu$, and such that there is an isomorphism
between $L^\infty(B,\nu)$ and the von Neumann algebra $\calH_\mu^\infty(\Gamma)$ of
all bounded $\mu$-harmonic functions on $\Gamma$ (see \S~\ref{subs:stat-harm-Poisson} below).
As a probability space with the non-singular $\Gamma$-action the Poisson boundary $(B,\nu)$ 
is uniquely determined by $(\Gamma,\mu)$, and can sometimes be realized on natural topological
$\Gamma$-spaces.
There is a large literature identifying Poisson boundaries and their dynamical properties
for many classes of groups, with the works of Furstenberg  \cites{Furst1, Furst2, Furst3}, 
Kaimanovich-Vershik \cite{KV}, Kaimanovich \cite{Kaim_Poisson}, Erschler \cite{Erschler} 
being some of the milestones.

\medskip

In this paper we are interested in understanding not only the Poisson boundary $(B,\nu)$ itself
but also the set of its measurable $\Gamma$-equivariant quotients.
Following Furstenberg we call them $\mu$-\textit{boundaries}.
Each of the $\mu$-boundaries $(B',\nu')$ is a $\mu$-stationary measurable $\Gamma$-space, 
and the collection $\Bnd(\Gamma,\mu)$ of all $\mu$-boundaries forms an order-theoretic lattice (see \S~\ref{subs:Bnd-lattice}).

Hereafter we will assume that $\mu$ on $\Gamma$ has finite entropy $H(\mu)<\infty$, where 
$H(\mu)=\sum -\mu(\gamma)\cdot\log\mu(\gamma)$.  (In fact, later we will impose a stronger
condition of $\mu$ having a finite exponential moment). 
The sequence $n^{-1}\cdot H(\mu^{*n})$, $n\in\bbN$, of normalized entropies 
converges to a limit $\hRW(\Gamma,\mu)$, often called \textit{Avez entropy}. 
It satisfies $0\le \hRW(\Gamma,\mu)\le H(\mu)$.
For every $\mu$-stationary space $(S,\eta)$, in particular for any $\mu$-boundary, 
the following formula defines a numerical invariant, often called \textit{Furstenberg entropy}, 
of the $\mu$-stationary space $(S,\eta)$ 
\begin{equation}\label{e:F-ent}
	h_\mu(S,\eta):=\int_{\Gamma} \int_S 
	-\log\frac{\dd g_*^{-1}\eta}{\dd\eta}\dd\eta \dd\mu. 
\end{equation}
The value $h_\mu(S,\eta)$ satisfies $0\le h_\mu(S,\eta)\le \hRW(\Gamma,\mu)$, and as
was shown by Kaimanovich-Vershik \cite{KV} the Poisson boundary attains the maximal
entropy $h_\mu(B,\nu)=\hRW(\Gamma,\mu)$.   
The Poisson boundary is the only $\mu$-boundary attaining this value.
Furthermore, the entropy functional
$(B',\nu')\ \mapsto\ h_\mu(B',\nu')$ is a monotonic functional on the space
of all $\mu$-stationary spaces, and it is \textit{strictly monotonic functional}
on the order-theoretic lattice $\Bnd(\Gamma,\mu)$.

\medskip

One important problem that has attracted attention of researchers is the possible range of entropy values
$h_\mu(S,\eta)$ within the interval $[0,\hRW(\Gamma,\mu)]$, where $(S,\eta)$ ranges over 
$\mu$-stationary spaces. An early surprising result in this direction by Nevo 
(see \cite{NZ1}) states that for a group $\Gamma$ with Kazhdan's property (T) there is a gap 
$(0,h_0)$ near $0$ in entropy values. 
Namely, there exists $h_0>0$, depending on $\Gamma$ and $\mu$, so that for 
any $\mu$-stationary space $(S,\eta)$ one has 
\[
	h_\mu(S)\in \{0\}\cup[h_0,\hRW(\Gamma,\mu)].
\]
The converse was later shown by Bowen--Hartman--Tamuz (\cite{BHT}): only groups with property (T) 
have such a gap near $0$.
On the other hand for finitely generated free groups $\Gamma=F_d$ 
and $\mu$ equidistributed on the generators
Bowen \cite{Bowen} showed that all values in $[0,\hRW(F_d,\mu)]$
are realized by entropies of ergodic $\mu$-stationary measures.
Entropy realization problems by $\mu$-boundaries
and Poisson boundaries of quotient groups were further studied by
Hartman and Tamuz \cite{HT}, and Tamuz and Zheng \cite{TZh}.
Our results below focus on a slightly different problem and
are using different type of construction. 
Our general goal is to address the following question. 

\begin{problem}
	Explore possible structures of the order-theoretic lattice $\Bnd(\Gamma,\mu)$
	equipped with the strictly monotonic entropy functional 
	\[
		h_\mu:\Bnd(\Gamma,\mu) \to \left[0,\hRW(\Gamma,\mu)\right].
	\]
\end{problem} 

\medskip

The following is a very special example, where the structure of $\Bnd(\Gamma,\mu)$ is completely understood.

\begin{example}\label{E:higherrank}
	Let $\Gamma$ be an (irreducible) lattice in a (semi)-simple Lie group $G$ of rank $r\ge 2$,
	e.g. $\SL_d(\mathbf{R})$ with $d\ge 3$ (here $r=d-1$).
	Furstenberg \cite{Furst3} showed that there exists a generating probability measure $\mu$ on $\Gamma$ 
	such that the unique $\mu$-stationary 
	measure $\nu$ on $G/P$ is in the $G$-invariant measure class.
	There exists such $\mu$ with finite entropy \cite{Kaim_discretization}, so that $(G/P,\nu)$ is 
	the Poisson boundary for $(\Gamma,\mu)$. 
	By the famous Factor Theorem of Margulis (\cite{Margulis_factor}, \cite{Margulis_book}*{Theorem IV, (2.11)})
	measurable $\Gamma$-equivariant quotients of $(G/P,\nu)$ are $G$-equivariant,
	and therefore have the form $G/P_\Theta$, where $P<P_\Theta<G$ is a parabolic subgroup
	that corresponds to a subset $\Theta\subset\{\alpha_1,\dots,\alpha_r\}$ of simple roots
	of the Lie algebra $\mathfrak{g}$ of $G$, with $G=P_\emptyset$ and $P=P_{\{1,\dots,r\}}$.
	Thus the lattice $\Bnd(\Gamma,\mu)$ is an $r$-dimensional cube $\{0,1\}^r$.
	Moreover, it is possible to show that there exist positive numbers $h_1,\dots, h_r>0$ so that 
	\[
	    h_\mu(G/P_\Theta)=\sum_{i\in\Theta} h_i.
	\]
	In particular, the following inclusion/exclusion formula 
	\begin{equation}\label{e:incexcl}
	    h_\mu(B_1)+h_\mu(B_2)=h_\mu(B_1\vee B_2)+h_\mu(B_1\wedge B_2)
	\end{equation}
	holds in this example. Similar statements apply to $S$-arithmetic lattices, and groups $\Gamma$ with admitting cofinite action on exotic $\tilde{A}_2$ buildings.
\end{example}

\medskip
\subsection{Main results}

Our results will apply to free groups $\Gamma=F_n$ and surface groups 
$\Gamma=\pi_1(\Sigma_g)$, and we will consider measures $\mu$ on $\Gamma$ 
that \textit{generate} $\Gamma$ as a semi-group and have a \textit{finite exponential moment},
meaning 
\begin{equation}\label{e:fin-exp-moment}
    \sum_{\gamma\in\Gamma} \mu(\gamma)\cdot e^{\epsilon\cdot |\gamma|}<+\infty.
\end{equation}
for some $\epsilon>0$ and some word metric $|\gamma|$ 
on the finitely generated group $\Gamma$.

\begin{mthm}\label{T:main-realization}
    Let $\Gamma=F_n$ be a free group on $n\ge 3$ generators
    or a fundamental group $\Gamma=\pi_1(\Sigma_g)$ of a surface of genus $g\ge 2$, 
    and $\mu$ a probability measure on $\Gamma$ that generates $\Gamma$ 
    as a semi-group
    and has a finite exponential moment.
    Then there exists $h_1=h_1(\Gamma,\mu)>0$ and a family 
    \[
        \{B_{s,t}\}_{s\in[0,1], t\in (0,h_1)}
        \subset \Bnd(\Gamma,\mu)
    \]
    of distinct $\mu$-boundaries such that:
    \begin{itemize}
        \item For every $0<t<h_1$ one has $h_\mu(B_{s,t})=t$.
        \item The $\Gamma$-action on each $B_{s,t}$ is essentially free.
        \item Every $B_{s,t}$ is minimal, in the sense that it has 
        no non-trivial $\Gamma$-quotients. In particular
        $B_{s,t}\wedge B_{s',t}=\{*\}$
        unless $s=s'$.
    \end{itemize} 
\end{mthm}
In particular, this result shows that all entropies at the bottom of the range, 
namely in the interval $[0,h_1]$, are realized by $\mu$-boundaries. 
In fact, each such value is attained by $2^{\aleph_0}$-many unrelated 
boundaries. %, indicates complexity of $\Bnd(\Gamma,\mu)$.

One way of constructing a $\mu$-boundary for $\Gamma$ is to consider
a (non-amenable) quotient group $\Gamma'=\Gamma/N$ of $\Gamma$,
taking the push-forward measure $\mu'$ on $\Gamma'$ coming from $\mu$.
Then the Poisson boundary $(B',\nu')$ of $(\Gamma',\mu')$, viewed as a $\Gamma$-space,
is a $\mu$-boundary (it can be identified as the space of $N$-ergodic components
for the $\Gamma$-action on $(B,\nu)$).
We point out that not that none of the $\mu$-boundaries $B_{s,t}$ above is of this form, 
because the $\Gamma$-action on $B_{s,t}$ is essentially free.

\bigskip

The next result is another illustration of the complexity of the 
order-theoretic lattice $\Bnd(\Gamma,\mu)$ graded by $h_\mu$. 
It shows a presence of many \textit{cubes} $\{0,1\}^d$ embedded in 
$\Bnd(\Gamma,\mu)$ and the validity of formula (\ref{e:incexcl}) 
at least in some examples. 

\begin{mthm}\label{T:main-cubes}
    Let $\Gamma=F_n$ be a free group on $n\ge 5$ generators
    or a fundamental group $\Gamma=\pi_1(\Sigma_g)$ of a surface of genus $g\ge 3$, 
    and $\mu$ a probability measure on $\Gamma$ that generates $\Gamma$ as a semi-group
    and has a finite exponential moment.
    Then there exists $c>0$ such that for every $d\in\bbN$ and
    parameter set
    \[
        T_d=\setdef{\textbf{t}=(t_1,\dots, t_d) }{0<t_i<Cd^{-a}}
    \] 
    there exists a family 
    $\{C_{s,\textbf{t}}\}_{s\in[0,1], t\in T_d}$ of 
    $\mu$-boundaries such that the sublattice of 
    quotients of each $C_{s,\textbf{t}}$ 
    is a $d$-cube $\{C^{\omega}_{s,\textbf{t}}\}_{\omega\in \{0,1\}^d}$, where
    \begin{itemize}
        \item $C^{(1,\dots,1)}_{s,\textbf{t}}=C_{s,\textbf{t}}\ $ 
        and $\ C^{(0,\dots,0)}_{s,\textbf{t}}=\{*\}$,
        \item $C^{\omega}_{s,\textbf{t}}\vee C^{\omega'}_{s,\textbf{t}}
        =C^{\omega\vee \omega'}_{s,\textbf{t}}\ $
        and
         $\ C^\omega_{s,\textbf{t}}\wedge C^{\omega'}_{s,\textbf{t}}
         =C^{\omega\wedge \omega'}_{s,\textbf{t}}$,
        \item $h_\mu(C^\omega_{s,\textbf{t}})=\sum_{i=1}^d \omega_it_i$,
        \item $\Gamma$-action on each $C^\omega_{s,\textbf{t}}$ is essentially free.
     \end{itemize}
\end{mthm}

\medskip

\subsection{Our approach}

Consider the simple real Lie group $G=\SO(3,1)$ that is essentially $\Isom(\bbH^3)$,
and consider the space $\Hom^Z(\Gamma,G)$ of homomorphisms 
\[
    \rho:\Gamma\ \overto{}\ G
\]
with Zariski dense image. In fact, since $G$ has rank one, and $\bbH^3$ is hyperbolic, 
being Zariski dense is equivalent to saying $\rho(\Gamma)$ is unbounded and non-elementary.
In these circumstances the measure $\mu$ on $\Gamma$ defines a random walk on $G$
that has a unique stationary measure $\nu_\rho\in\Prob(\partial\bbH^3)$ and that
$(\partial\bbH^3,\nu_\rho)$ is a $\mu$-boundary. We shall denote it $(B_\rho,\nu_\rho)$.

The idea is to vary $\rho\in \HomZd(\Gamma,G)$ in order to construct a large family 
of $\mu$-boundaries $B_\rho\in\Bnd(\Gamma,\mu)$.
It seems plausible that the entropy $h_\mu(B_\rho)$ would depend continuously 
on the representation $\rho$ in all cases.
While the associated $\mu$-stationary measures $\nu_\rho$ do vary continuously 
in the weak-* topology on the sphere $S^2=\partial\bbH^3$,
this fact alone does not seem to be sufficient to show that the associated 
entropy varies continuously. 
Hence  we focus on the situations where the stationary measure $\nu_\rho$
is in the $G$-invariant measure class, and has a $\log$-integrable 
Radon-Nikodym derivative 
with respect to the $\SO(3)$-invariant measure on the $2$-sphere $S=\partial\bbH^3$.
In these cases we are able to prove that the entropy $h_\mu(B_\rho,\nu_\rho)$ varies
continuously in $\rho$. 

The problem of achieving \textit{regularity} for the stationary measure $\nu_\rho$ 
in the Lebesgue class, 
has been recently considered by Bourgain \cite{Bourgain} and 
Benoist--Quint \cite{BQ_smooth}.
We revisit and simplify these approaches in Section~\ref{sec:regularity}. 
For Theorem~\ref{T:main-cubes} we need to study stationary measures 
for $\rho_1\times\cdots\times\rho_d:\Gamma\to G=\PSL_2(\bbC)^d$. 
The need to achieve regularity for the stationary measures forces us to work in the 
low boundary entropy region $(0,h_d)\subset [0,\hRW(\Gamma,\mu)]$.

\subsection*{Acknowledgements}
This work was supported in part by the NSF Grant DMS 2005493 and by the BSF Grant 2018258.

\section{Notations and some Preliminary facts}

In this section we fix some terminology and notations and recall some background facts that we rely on in the rest of the paper.

\subsection{Stationary spaces, $\mu$-harmonic functions, and Poisson boundary}
\label{subs:stat-harm-Poisson}

A probability space $(S,\eta)$ with a measurable, measure-class preserving action 
of $\Gamma$ is called
$\mu$-\textit{stationary} if $\mu*\eta=\eta$, where 
\[
    \mu*\eta:=\sum_{g\in\Gamma} \mu(g)\cdot g_*\eta.
\]
Given a $\mu$-stationary $\Gamma$-space $(S,\eta)$ and a function $f\in L^\infty(S,\eta)$ 
define a bounded function $F\in\ell^\infty(\Gamma)$ by the formula
\begin{equation}\label{e:PTeta}
    F(g)=\int_S f\dd g_*\eta=\int_S f(x)\cdot \frac{\dd g_*\eta}{\dd\eta}(x)\dd\eta(x) 
\end{equation}
The assumption that $\eta$ is $\mu$-stationary implies that $F$ satisfies the following 
$\mu$-\textit{mean value property}:
\begin{equation}\label{e:muMVP}
    F(g)=\sum_{g'\in\Gamma} \mu(g')\cdot F(gg')\qquad (g\in\Gamma).
\end{equation} 
Functions satisfying this property are called $\mu$-\textit{harmonic},
and 
\[
    \calH^\infty_\mu(\Gamma):=\setdef{F\in\ell^\infty(\Gamma)}{F\ \textrm{satisfies\ } (\ref{e:muMVP})}
\]
denotes the space of all bounded $\mu$-harmonic functions on $\Gamma$.

%%%% COMMENTED OUT FOR BREVITY
%    This is a Banach subspace of $\ell^\infty(\Gamma)$ with an isometric $\Gamma$-action 
%   by left translations.
%  In fact, one can convert $\calH^\infty_\mu(\Gamma)$ into a commutative von Neumann algebra by
% introducing the following multiplication operation 
% (that is different from the one of $\ell^\infty(\Gamma)$)
% \[
%     (f_1*f_2)(g)=\lim_{n\to\infty} \sum_{g'\in\Gamma} \mu^{*n}(g')\cdot f_1(gg')\cdot f_2(gg').
% \]
Formula (\ref{e:PTeta}) defines a positive, normalized, linear map
\[
    \FP_S:L^\infty(S,\eta) \overto{} \calH^\infty_\mu(\Gamma)
\]
that is sometimes called Furstenberg--Poisson Transform.
There exists a unique $\mu$-stationary $\Gamma$-space $(B,\nu)$ for which 
Furstenberg--Poisson Transform
\[
   \FP_B: L^\infty(B,\nu) \overto{} \calH^\infty_\mu(\Gamma)
\]
is an isomorphism. 
In fact, this $(B,\nu)$ is the von Neumann spectrum of the commutative von Neumann algebra 
$L^\infty(B,\nu)\cong \calH^\infty_\mu(\Gamma)$, where $\nu$ corresponds to the evaluation  
functional $f\mapsto f(e)$ on $\calH^\infty_\mu(\Gamma)$.
This special stationary $\Gamma$-space $(B,\nu)$ is called the \textit{Poisson} 
or \textit{Furstenberg--Poisson boundary} of $(\Gamma,\mu)$.

A $\mu$-stationary $\Gamma$-space $(S,\eta)$ is called a $\mu$-\textit{boundary} if 
$\FP_S$ is an embedding. The Poisson boundary $(B,\nu)$ is the \textit{universal} $\mu$-boundary:
for any $\mu$-boundary $(B',\nu')$ there is a unique $\Gamma$-equivariant measurable quotient map
\[
    p:(B,\nu)\overto{} (B',\nu')\qquad p_*\nu=\nu'
\]
corresponding to the embedding $L^\infty(B',\nu')\to L^\infty(B,\nu)$.

\medskip

\subsection{The lattice of boundaries}
\label{subs:Bnd-lattice}

We are interested in understanding the partially ordered set (POSet) $\Bnd(\Gamma,\mu)$ 
of all $\mu$-boundaries $(B',\nu')$, where $(B',\nu')\ge (B'',\nu'')$ when there exists 
a $\Gamma$-equivariant measurable map $q:(B',\nu')\to(B'',\nu'')$ with $q_*\nu'=\nu''$.
In this POSet $(B,\nu)$ is the \textit{maximal boundary} and the singleton $\{*\}$ 
is the \textit{minimal boundary}.
In view of the universality of $(B,\nu)$ there is a a correspondence between the following objects
\begin{itemize}
    \item $\mu$-boundary $(B',\nu')$ (expressible as $p:B\to B'$)
    \item $\Gamma$-invariant von Neumann sub-algebra in $\calH^\infty_\mu(\Gamma)$, namely 
    the image under $\FP_{B'}$ of $L^\infty(B',\nu')$.
    \item $\Gamma$-invariant $\nu$-complete sub-$\sigma$-algebra of $\calB_B$ (expressible as $p^{-1}\calB_{B'}$).
\end{itemize}
The partial order is most transparent in the last two points of view as 
it corresponds to inclusions of algebras.
It also becomes clear from these latter perspectives that $\Bnd(\Gamma,\mu)$ is an 
\textit{order-theoretic lattice}.
More precisely. 
\begin{lemma}\label{L:boundary-lattice}
    Let $\{B_i\}_{i\in I}$ be a family of $\mu$-boundaries in $\Bnd(\Gamma,\mu)$. 
    Then there exists a unique $\mu$-boundary $C=\wedge_{i\in I} B_i$, called the meet of $\{B_i\}$,
    and a unique $\mu$-boundary $D=\vee_{i\in I} B_i$, called the join of $\{B_i\}$, with the following properties:
    \begin{itemize}
        \item $C\le B_i$ for all $i\in I$, and if $C'\le B_i$ for all $i\in I$, then $C'\le C$.
        \item $D\ge B_i$ for all $i\in I$, and if $D'\ge B_i$ for all $i\in I$, then $D'\ge D$.
    \end{itemize}
\end{lemma}
\begin{proof}
    Let $\calB_i$ denote the $\nu$-complete $\Gamma$-invariant sub-$\sigma$-algebra of $\calB_B$ on $B$
    corresponding to $B\to B_i$. Then the intersection
    \[
        \calC:=\bigcap_{i\in I} \calB_i
    \]
    is the $\sigma$-algebra defining $\wedge_{i\in I} B_i$, and the join
    \[
        \calD:=\bigvee_{i\in I}\calB_i
    \]
        is the $\sigma$-algebra defining $\vee_{i\in I} B_i$.
\end{proof}

% \begin{defn}
%     A family $\{B_i\}_{i\in I}$ of $\mu$-boundaries is \textit{pairwise disjoint} if $B_i\wedge B_j=\{*\}$ for all $i\ne j$ in $I$.
% \end{defn}

\subsection{Entropy}

Let us now assume that $\mu$ on $\Gamma$ has a finite entropy $H(\mu)<\infty$.
Here $H(\mu)$ denotes the entropy of the countable partition of the probability space $(\Gamma,\mu)$
into atoms:
\begin{equation}\label{e:mu-ent}
	H(\mu):=\sum_{\gamma\in \Gamma} -\mu(\gamma)\cdot\log \mu(\gamma) <+\infty
\end{equation}
with the convention that $0\cdot \log 0=0$.
For measures with finite entropy, there is a notion of asymptotic entropy 
of the $\mu$-random walk on $\Gamma$,
introduced by Avez \cite{Avez} and studied by Kaimanovich--Vershik \cite{KV}, as the limit
\[
	\hRW(\Gamma,\mu):=\inf_{n\ge 1}\frac{1}{n} H(\mu^{*n})=\lim_{n\to\infty}\frac{1}{n} H(\mu^{*n})
\]
where $\mu^{*n}$ denotes the $n$th convolution power of $\mu$. The existence of the limit
is a consequence of the subadditivity $H(\mu^{*(n+m)})\le H(\mu^{*n})+H(\mu^{*m})$, 
that follows from the general properties of the entropy function and the 
fact that the product map $\Gamma\times\Gamma\to \Gamma$, 
$(g_1,g_2)\mapsto g_1\cdot g_2$, maps $\mu^{*m}\times\mu^{*n}$ to $\mu^{*(m+n)}$.

\medskip

Let $(S,\eta)$ be a $\mu$-stationary $\Gamma$-space. 
The boundary entropy, also called Furstenberg entropy, is defined by
\[
	h_\mu(S):=\int_{\Gamma} \int_S 
	-\log\frac{\dd g_*^{-1}\eta}{\dd\eta}\dd\eta \dd\mu 
\]
We recall the basic "soft" properties (cf. \cite{Furman}*{Chapter 3}):
\begin{itemize}
    \item $h_\mu(S)\ge 0$ with $h_\mu(S)=0$ iff $\Gamma\acts (S,\eta)$ is measure-preserving.
    \item $h_\mu(S)\le \hRW(\Gamma,\mu)$ and this maximum is achieved by the Poisson boundary: 
    $h_\mu(B)=\hRW(\Gamma,\mu)$.
    \item $h_\mu$ is monotonic: if $S_1\to S_2$ then $h_\mu(S_1)\ge h_\mu(S_2)$
    and the inequality is strict unless $S_1\to S_2$ is relatively probability measure preserving.
    \item Any $\mu$-stationary $(S,\eta)$ is also $\mu^{*n}$-stationary for $n\in\bbN$ and
    \[
        h_{\mu^{*n}}(S)=n\cdot h_\mu(S).
    \]
\end{itemize}
Specializing to $\mu$-boundaries, we remark that relatively probability 
measure preserving map  $B_1\to B_2$ between $\mu$-boundaries has to be an isomorphism. 
This implies that $h_\mu$ is a \textit{strictly monotonic} map on $\Bnd(\Gamma,\mu)$ 
taking values in $[0,\hRW(\Gamma,\mu)]$.

\subsection{Integrability conditions on the measure $\mu$}\label{subs:measure}
Let $\Gamma$ be a finitely generated group with a generating set $A=\{a_1,\dots,a_n\}$.
The associated word length $|-|_A$ on $\Gamma$ is defined by
\[
    |\gamma|_A:=\min\setdef{n}{\gamma\in (A\cup A^{-1})^n}
\]
with $|e|_A=0$ by convention. 

Let $\mu$ be a probability measure on $\Gamma$.
We say  $\mu$ is \textit{generating $\Gamma$ as a group} if %%%%%%
$\Gamma=\bigcup_{n=1}^\infty(\supp(\mu)\cup \supp(\mu)^{-1})^n$.
We say that $\mu$ \textit{generates $\Gamma$ as a semi-group} if 
$\Gamma=\{e\}\cup \bigcup_{n=1}^\infty(\supp(\mu))^n$.
Clearly the latter condition implies the former. 
A probability measure $\mu$ on $\Gamma$ has \textit{finite first moment} if 
\begin{equation}\label{e:mu-L1}
    \sum_\gamma \mu(\gamma)\cdot |\gamma|_A<\infty.
\end{equation}
We say that $\mu$ has \textit{finite exponential moment} 
if there exists $\epsilon>0$ 
so that (\ref{e:fin-exp-moment}) holds, i.e.
\[
    \sum_\gamma \mu(\gamma)\cdot e^{\epsilon\cdot|\gamma|_A}<\infty.
\]
For a measure $\mu$ on a finitely generated group $\Gamma$ 
the conditions of having finite support,
having finite exponential moment, 
having finite first moment, and having finite entropy are progressively weaker.
We shall need the following:
%(cf. \cite{}*{} for the latter implication).

\begin{lemma}\label{L:conv-of-exp}
    If $\mu$ has finite exponential moment, then any convolution power $\mu^{*n}$
    and any finite convex combination of convolution powers 
    has a finite exponential moment.
\end{lemma}

\begin{proof}
    It suffices to show that if $\mu_1$ and $\mu_2$ both satisfy (\ref{e:fin-exp-moment})
    then so does $\mu_1*\mu_2$.
    Consider the positive integrable functions 
    $f_i\in \ell^1(\Gamma)$ given by 
    $f_i(\gamma)=\mu_i(\gamma)\cdot e^{\epsilon |\gamma|_A}$ for $i=1,2$.
    Consider their convolution $f_1*f_2$ and estimate
    \[
        \begin{split}
            f_1*f_2(\gamma)&=\sum_{x\cdot y=\gamma} f_1(x)\cdot f_2(y)
        =\sum_{x\cdot y=\gamma}\mu_1(x)\cdot e^{\epsilon |x|_A}\cdot \mu_2(y)
        \cdot e^{\epsilon |y|_A}\\
        &=\sum_{x\cdot y=\gamma}\mu_1(x)\cdot \mu_2(y)\cdot e^{\epsilon (|x|_A+|y|_A)}\\
        &\ge \sum_{x\cdot y=\gamma}\mu_1(x)\mu_2(y)\cdot e^{\epsilon |\gamma|_A}
        =\mu_1*\mu_2(\gamma)\cdot e^{\epsilon |\gamma|_A}
        \end{split}
     \]
    Since its is well known (Fubini argument) that $\|f_1*f_2\|_1\le \|f_1\|_1\cdot \|f_2\|_1$,
    we deduce that $\mu_1*\mu_2$ has finite exponential moment, in fact 
    with respect to the same exponent $\epsilon>0$.
    The statement about convex combinations is obvious.
\end{proof}

\medskip

\subsection{Semi-simple Lie groups}\label{subs:ss}
Let $G$ be a connected semi-simple real Lie group with finite center, 
$K<G$ a maximal compact subgroup,
$P<G$ a minimal parabolic, and let $S=G/P$ be the flag variety. 
The Iwasawa decomposition $G=K\cdot P$ shows that $K$ acts transitively on $S$,
hence $S=K/M$ where $M=K\cap P$. We denote by $m$ the unique $K$-invariant 
probability measure on $S$.

% We also have a Cartan decomposition $G=K\cdot \exp(\mathfrak{a}^+)K$
% and the associated Cartan projection $p:G\overto{} \mathfrak{a}^+$ (cf. \cite{BQ_book}*{}).
 
We will need to describe a "norm" $\calN:G\overto{}[1,\infty)$ satisfying 
\[
    \calN(g_1g_2)\le \calN(g_1)\cdot\calN(g_2)
\]
and $\calN(g)=1$ iff $g\in K$.
There are several natural options to do so, including the following:
\[
    \calN_{G/K}(g)=e^{d_{G/K}(K,gK)},\qquad
    \calN_{\Ad}(g)=\|\Ad g\|,\qquad
    \NRN(g)=\|\frac{\dd g_*m}{\dd m}(x)\|_\infty
\]
where $d_{G/K}$ is the $K$-invariant distance on the symmetric space $G/K$ associated
with the Killing form, and $\|-\|:\Ad G\to [0,\infty)$ is the operator norm.
These norms are equivalent in the following sense: there exist positive constants
$c_1,c_2,c_3>0$ so that for all $g\in G$
\[
    \NRN(g)\le \calN_{G/K}(g)^{c_1}\le \calN_{\Ad}(g)^{c_2}\le \NRN(g)^{c_3}
\]
We will be using the $\NRN$ norm, and the above relation allow one
to translate these statements to other norms.
\medskip 

Let $\Gamma$ be a group generated by a finite set $A$.
Let $G$ be a semi-simple real Lie group, and consider the space
\[
    \Hom(\Gamma,G)
\]
of all representations $\rho:\Gamma\overto{} G$.
Since any representation $\rho:\Gamma\overto{} G$ is determined 
by its values $\rho(a)$ on the generators $a\in A$, the space
$\Hom(\Gamma,G)$ can be viewed as an algebraic subvariety of $G^A$;
called the \textit{representation variety}.
The Hausdorff topology on $\Hom(\Gamma,G)$ is inherited from $G^A$;
we have $\rho_i\to \rho$ iff $\rho_i(a)\to \rho(a)$ for all $a\in A$.
Denote
\[
    \HomZd(\Gamma,G):=\setdef{\rho\in\Hom(\Gamma,G)}{\rho(\Gamma)\ \ 
    \textrm{is\ Zariski\ dense\ in\ } G}.
\]

\medskip

\subsection{Random walks}\label{subs:RW-on-G}

The study of random walks on semi-simple Lie groups going back to the works of 
Furstenberg \cites{Furst1, Furst2, Furst3}, Guivarc'h and Raugi \cite{GR}, 
Gol'dsheid and Margulis \cite{GM} (see Bougerol--Lacroix \cite{BL}, 
Furman \cite{Furman}, Benoist--Quint \cite{BQ_book}) lead to the following statement 
that we phrase in a form convenient for our applications.

\begin{theorem}\label{T:Zd-bnd}
    Let $\mu$ be a generating probability measure on a countable group $\Gamma$.
    Let $G$ be a semi-simple Lie group and $\rho:\Gamma\overto{} G$ be a 
    homomorphism with Zariski dense image. 
    Then the continuous $\Gamma$-action on the compact space $S=G/P$,
    given by $\gamma:gP\mapsto \rho(\gamma)gP$, 
    has a unique stationary measure $\nu_\rho\in\Prob(S)$
    and there is a measurable map $\ \xi:(\Gamma^\bbN,\mu^\bbN)\overto{} S$ such that
    \begin{equation}\label{e:RWconv}
          \rho(\gamma_1\gamma_2\cdots \gamma_n)_*\nu_\rho\ \overto{}\ 
        \delta_{\xi(\gamma_1,\gamma_2,\dots)}
    \end{equation}
    in the weak-* topology for $\mu^\bbN$-a.e. sequence 
    $(\gamma_1,\gamma_2,\dots)\in \Gamma^\bbN$.
\end{theorem}
We claim that $(S,\nu_\rho)$ is a $\mu$-boundary of $\Gamma$ that we will denote by 
$(B_\rho,\nu_\rho)$ to emphasize the dependence on $\rho\in\HomZd(\Gamma,G)$.
Indeed, let $\bbP$ be the push-forward of the product measure $\mu^\bbN$ under the map 
\[
    \Gamma^\bbN\overto{} \Gamma^\bbN\qquad  
    (\gamma_1,\gamma_2,\gamma_3,\dots)\mapsto 
    (\gamma_1, \gamma_1\gamma_2, \gamma_1\gamma_2\gamma_3,\dots).
\]
Then $\bbP$ denotes the distribution of \textit{paths} of the $\mu$-random walk
whose \textit{increments} are given by $\mu^\bbN$. 
The limit in (\ref{e:RWconv}) shows that $\xi$ descends to
a map 
\[
    \xi_0:(\Gamma^\bbN,\bbP)\to (S,\nu_\rho)
\]
that is invariant under 
the shift in the path space, and satisfies $\xi_0(\gamma.\omega)=\rho(\gamma)\xi_0(\omega)$
for $\bbP$-a.e. path $\omega$.
The Poisson boundary $(B,\nu)$ is defined as the space of ergodic components 
of $(\Gamma^\bbN,\bbP)$
under this shift, so $\xi_0$ further descends to a measurable map
\[
    \xi_\rho:(B,\nu)\overto{}(S,\nu_\rho)\qquad \xi_\rho(\gamma.b)=\rho(\gamma).\xi_\rho(b)
\]
To emphasize the dependence on $\rho\in\HomZd(\Gamma,G)$ we denote $(S,\nu_\rho)$ by $(B_\rho,\nu_\rho)$. 
\begin{corollary}\label{C:construction0}
    Let $\mu$ be a generating probability measure on a countable group $\Gamma$.
    Let $G$ be a semi-simple real Lie group. Then there is a map
    \[
        \HomZd(\Gamma,G)\ \overto{} \Bnd(\Gamma,\mu),
        \qquad \rho\ \mapsto\ (B_\rho,\nu_\rho)
    \]
    described by Theorem~\ref{T:Zd-bnd}.
\end{corollary}

\medskip

\subsection{Representations}\label{subs:reps}

Let $G$ and $S=G/P$ with $K$-invariant $m\in \Prob(S)$ as above.
The measure-class preserving $G$-action on $(S,m)$ gives rise to two 
$G$-representations on the Hilbert space $\calH=L^2(S,m)$: the 
\textit{quasi-regular} unitary representation 
\begin{equation}\label{e:quasi-reg-m}
    \pi:G\overto{} U(\calH),\qquad 
    (\pi(g)f)(x)=\left(\frac{\dd g^{-1}_* m}{\dd m}(x)\right)^{1/2}\cdot f(g^{-1}.x)
\end{equation}
and the \textit{bounded representation} given by
\begin{equation}\label{e:direct}
    T:G\overto{} B(\calH),\qquad
    (T_gf)(x)=f(g^{-1}.x).
\end{equation}
The dual operators are given by
\[
    \begin{split}
    &(\pi(g)^*f)(x)=(\pi(g^{-1})f)(x)
    =\left(\frac{\dd g_* m}{\dd m}(x)\right)^{1/2} f(g.x)\\
     &(T_g^*f)(x)=\frac{\dd g_*m}{\dd m}(x)\cdot f(g.x).
    \end{split}
\]
The fact that $T_g$ are bounded operators on $\calH=L^2(S,m)$ 
follows from the fact that each $g\in G$ has bounded Radon-Nikodym 
derivative on $(S,m)$. 
In fact, it would be convenient to estimate
\begin{lemma}
    For each $g\in G$ the operator norm $\|T_g\|$ satisfies 
    \[
        \|T_g\|= \|\frac{\dd g^{-1}_*m}{\dd m}\|_{2,m}\le \NRN(g).
    \]
\end{lemma}

\medskip

\subsection{Spectral Gap}\label{subs:spec-gap}

Let $\calH$ be a Hilbert space, and $U_1,\dots, U_k$ be some unitary
operators on $\calH$. The average operator
\[
    P=\frac{1}{2k}(U_1+\dots+U_k+U_1^*+\dots +U_k^*)
\]
is a self-adjoint contraction: $P=P^*$ and $\|P\|\le 1$.
We are interested in situations where $\{U_1,\dots,U_k\}$ are such that
$\|P\|<1$, in which case we call the quantity  
\[
    \kappa=1-\|P\|
\]
the \textit{spectral gap}.

Let $S=G/P=K/M$ with the $K$-invariant probability measure $m$,
the space $L^2(S,m)$ and 
\[
    L^2_0(S,m)=\setdef{f\in L^2(S,m)}{\int f\dd m=0}.
\]
The $K$ acts on $L^2_0(S,m)$ by unitary operators.
Given a finite set $C=\{c_1,\dots,c_k\} \subset K$ 
we have the associated unitary operators $U_1,\dots, U_k$ on $L^2_0(S,m)$ 
and denote by 
\[
    \kappa_C:=1-\sup\setdef{\left\|\frac{1}{2k}\sum_{i=1}^k (f\circ c_i+f\circ c_i^{-1})\right\|}{f\in L^2_0(S,m),\ \|f\|=1}
\]
the corresponding spectral gap.

Deep results of Drienfeld \cite{Drinfeld}, Lubotzky--Phillips--Sarnak 
\cites{LPS1, LPS2},
Bourgain--Gamburd \cite{BG} show that there exist finite subsets $C=\{c_1,c_2,\dots,c_k\}$ 
in $\SO(3)$ with a spectral gap $\kappa_C>0$ for the action on the sphere $S^2=\SO(3)/\SO(1)$.
In fact, with the exception of the circle any simple compact Lie group $K$ admits such a collection of elements (\cite{BdS}), but we interested in the smallest example of $\SO(3)$. 
The work of Lubotzky--Phillips--Sarnak  gives optimal spectral gap.

\begin{theorem}[\cites{LPS1, LPS2}]\label{T:LPS}
    For the case $K=\SO(3)$ acting on the two sphere $S=S^2$ there exist 
    $C=\{c_1,\dots,c_k\}\subset \SO(3)$ with $\kappa_C>0$.
    In fact, if $p=2k-1$ is a prime with $p\equiv1\mod 4$, there exist $C=\{c_1,\dots,c_k\}$ so that
    \[
        \kappa_C=\frac{\sqrt{2k-1}}{k}
    \]
    and $\{c_1,\dots,c_k\}$ are free generators of a free subgroup in $\SO(3)$.
\end{theorem}

\section{Regularity of stationary measures}
\label{sec:regularity}

The purpose of this section is to describe an explicit argument that allows 
us to construct a subset 
\[
    \HomZdsg(\Gamma,G)\subset \HomZd(\Gamma,G)
\]
of representations $\rho:\Gamma\overto{} G$ with a unique $\mu$-stationary 
measure $\nu_\rho$ on $S$ with $\nu_\rho\sim m$.
The basic idea is inspired by the work of Benoist--Quint \cite{BQ_smooth} 
(see also earlier work of Bourgain \cite{Bourgain}). 
However our argument is more direct and thus is easier to quantify.

\medskip

Consider a finitely generated group $\Gamma$ with a generating set $A$,
and let $B\subset \Gamma$ be a finite subset, where we may assume $B\subset A$.
Let $G$ be a connected semi-simple real Lie group with trivial center, 
and $S=G/P=K/M$ with the $K$-invariant probability measure as in \S~\ref{subs:ss}.
In the Theorem~\ref{T:Lebesgue-stat} below we consider representations $\rho:\Gamma\overto{} G$ 
such that $\rho(B)$ is contained in the compact subgroup $K$ has a spectral 
gap 
\[
    \kappa_{\rho(B)}>0
\]
on $L^2_0(S,m)$, while at the same time $\rho(A)$ is sufficiently close to $K$, 
but is not contained in $K$. 
We use the following parameter $\alpha_\rho\ge0$ to gauge  closeness to $K$:
\begin{equation}\label{e:alpha-norm}
    \alpha_\rho:=\max_{g\in \rho(A)\cup \rho(A)^{-1}}\log\NRN(g).
\end{equation}
Note that $\alpha(\rho)=0$ iff $\rho(A)$ is contained in $K$.
Since $\Gamma$ is generated by $A$ the latter is equivalent to 
$\rho(\Gamma)\subset K$.

\begin{remark}\label{R:top-dense}
    Let $G$ be a connected real Lie group of rank one, 
    a maximal compact subgroup $K<G$ is maximal as a closed subgroup.
    Thus for $\rho\in \HomZdsg(\Gamma,G)$ with $\alpha_\rho>0$ 
    the image $\rho(\Gamma)$ is topologically dense in $G$.
    In particular this implies to $G=\SO(3,1)$ and $K=\SO(3)$.  
\end{remark}

\begin{theorem}[Good stationary measure]\label{T:Lebesgue-stat}\hfill{}\\
    Let $\Gamma$ be a group generated by a finite set $A$,
    and let $B$ be a finite subset in $\Gamma$.
    Let $\mu\in\Prob(\Gamma)$ be generating as a semi-group
    and having a finite exponential moment.
    This data determines a constant $\Csg>0$  so that the following holds.
    
    Let $G$ be a semi-simple Lie group $G$, $m$ be the $K$-invariant 
    probability measure on $S=G/P$,   $\rho:\Gamma\overto{} G$ a representation with
    Zariski dense image in $G$, such that $\rho(B)\subset K$ 
    has a positive spectral gap $\kappa_{\rho(B)}>0$ on $L^2_0(S,m)$, while
    \[
        0<\alpha_\rho< \Csg\cdot \kappa_{\rho(B)}. 
    \]
    Then the $\Gamma$-action $S$ has a unique $\mu$-stationary measure $\nu_\rho$, 
    this measure is in the $G$-invariant measure class, and has density
    \[
        \phi_\rho=\frac{\dd \nu_\rho}{\dd m}\qquad\textrm{with}\qquad
        \int_S \phi^2\dd m<\infty,\qquad \int_S \phi^{-\tau} \dd m<\infty 
    \]
    for some $\tau>0$. 
    In particular, $\log \phi_\rho \in L^p(S,m)$ for all $p\in [1,\infty)$. 
\end{theorem}

\medskip

\begin{notation}\label{N:Zdsg}
    We denote by $\HomZdsg(\Gamma,G;\mu,A,B)\subset \HomZd(\Gamma,G)$
    the set of representations satisfying conditions of Theorem~\ref{T:Lebesgue-stat}.
\end{notation}

We start with a general setting.
Let $(X,m)$ be a probability space, $L^2(X,m)$ denote the real Hilbert space of square integrable functions, 
$\mathbb{1}\in L^2(X,m)$ denote the constant $1$ function,
and its orthogonal compliment $\mathbb{1}^\perp$ by 
\[
    L^2_0(X,m)=\setdef{f\in L^2(X,m)}{\int_X f\dd m=0}.
\]
We shall also denote by $L^2_+(X,m)=\setdef{f\in L^2(X,m)}{f\ge 0}$ the cone of non-negative square 
integrable functions.
\begin{lemma}\label{L:regularity}
    Let $(X,m)$ be a probability space, $\kappa>0$ and $p\in (0,1)$ be constants, and 
    $P$ and $Q$ bounded operators on $L^2(X,m)$ satisfying: 
    \begin{enumerate}
        \item $L^2_+(X,m)$ is invariant under $P^*$ and $Q^*$,
        \item $P\mathbb{1}=Q\mathbb{1}=\mathbb{1}$,
        \item $\|P^*f\|\le (1-\kappa)\|f\|$ for all $f\in L^2_0(X,m)$,
        \item $\|Q^*\|<1+\kappa p/(1-p)$.
    \end{enumerate} 
    Then the dual $T^*$ of the operator $T=t P+(1-t)Q$ has a unique fixed vector 
    \[
        T^*\phi=\phi\qquad\textrm{with}\qquad \ip{\phi}{\mathbb{1}}=1
    \]
    And $\phi$ is positive: $\phi\in L^2_+(X,m)$.
\end{lemma}
\begin{proof}
    Consider the affine hyperplane 
    \[
        H=\mathbb{1}+L^2_0(X,m)=\setdef{f\in L^2(X,m)}{\ip{f}{\mathbb{1}}=1}
    \]
    and a closed convex subset $H_+=H\cap L^2_+(X,m)$ in it.
    Note that $H$ and $H_+$ are $T^*$-invariant.
    Indeed, $H$ is $T^*$-invariant due to assumption (2), because
    \[
        \ip{T^*f}{\mathbb{1}}=\ip{f}{T\mathbb{1}}=p\ip{f}{P\mathbb{1}}+(1-p)\ip{f}{Q\mathbb{1}}
        =\ip{f}{\mathbb{1}}.
    \]
    By assumption (1) the cone $L^2_+(x,m)$ is invariant under $T^*=p P^*+(1-p)Q^*$,
    hence $H_+$ is $T^*$-invariant as well.

    For any $f,g\in H_+$ we have $f-g\in L^2_0(X,m)$ and therefore, using assumption (3), we have
    \[
        \begin{split}
            \|T^*(f-g)\| &\le p \|P^*(f-g)\|+(1-p) \|Q^*(f-g)\|\\
            &\le p(1-\kappa) \|f-g\|+(1-p)\|Q^*\|\cdot \|f-g\|\\
            &=(p(1-\kappa)+(1-p)\|Q^*\|)\cdot \|f-g\|.
        \end{split}
    \]
    Assumption (4) implies that the constant $q=p(1-\kappa)+(1-p)\|Q^*\|$ 
    satisfies $q<1$ and therefore $T^*$ is a $q$-contraction
    on the complete spaces $H$ and $H_+$. 
    Hence $T^*$ has a unique fixed point $\phi\in H$
    and furthermore $\phi\in H_+$.
\end{proof}

We shall use the following elementary statement.

\begin{lemma}\label{L:elementary}
    Given $\epsilon>0$ for all $x\in [0,\epsilon/2]$ 
    \[
        \sup_{n\in\bbN}\ \frac{e^{n x}-1}{e^{\epsilon n}}
        \le \frac{2}{\epsilon}\cdot x.
    \]
\end{lemma}
\begin{proof}
    \[ 
    \frac{e^{n x}-1}{e^{\epsilon n}}= 
    \sum_{k=1}^\infty e^{-\epsilon n}\cdot\frac{n^k x^k}{k!} 
    \le \sum_{k=1}^\infty \frac{k!}{\epsilon^k n^k}\cdot\frac{n^k x^k}{k!}
    = \sum_{k=1}^\infty \left(\frac{x}{\epsilon}\right)^k
    = \frac{x}{\epsilon-x}\le \frac{2 x}{\epsilon}.
    \]
\end{proof}

\begin{proof}[Proof of Theorem~\ref{T:Lebesgue-stat}]
By assumption $\rho:\Gamma\overto{} G$ has Zariski dense image and 
$\supp(\mu)$ generates $\Gamma$
as a semi-group. In particular, conditions of Theorem~\ref{T:Zd-bnd} are satisfied.
Thus the $\Gamma$-action on the compact space $S=G/P$
\[
    \gamma: gP\ \mapsto\ \rho(\gamma)gP
\]
has a unique stationary measure $\nu\in\Prob(S)$.

Since $\supp(\mu)$ generates $\Gamma$ as a semi-group, 
for each $\gamma\in\Gamma$ some convolution power $\mu^{*n}$ gives positive 
weight to $\gamma$.
Therefore for any finite subset of $C\subset\Gamma$ there exists 
a convex combination of convolution powers of $\mu$
\[
    \bar\mu=\sum_{i=1}^{k}p_i\cdot \mu^{*n_i},\qquad 
    p_i>0,\qquad\sum_{i=1}^k p_i=1
\]
so that $\bar\mu\ge p\cdot\mu_0$ and $\bar\mu(a)\ge q$ for all $a\in A$,
where $\mu_0$ is the uniform distribution over the set $B\cup B^{-1}$,
and $p>0$, $q>0$.
Taking probability measure $\mu'=(1-p)^{-1}\cdot (\bar\mu-p\cdot \mu_0)$
on $\Gamma$ we get
\begin{equation}\label{e:power-of-mu}
    \bar\mu=p\cdot \mu_0+(1-p)\cdot \mu'
\end{equation}
Define two bounded operators $P,Q$ on $L^2(S,m)$ by
\[
    P=\sum_{\gamma} \mu_0(\gamma)\cdot T_{\rho(\gamma)},
    \qquad 
    Q=\sum_{\gamma} \mu'(\gamma)\cdot T_{\rho(\gamma)}
\]
and let 
\[
    T=p\cdot P+(1-p)\cdot Q=\sum_\gamma \bar\mu(\gamma)\cdot T_{\rho(\gamma)}.
\]
Note that $P\mathbb{1}=Q\mathbb{1}=\textbf{1}$.
Recall that $\rho(B)\subset K$ and $u_B$ is symmetric, so $P=P^*$. 
By our assumption $\rho(B)\subset K$ has a spectral gap $\kappa_{\rho(B)}$.
Thus
\[
    \|Pf\|\le (1-\kappa_{\rho(B)})\cdot \|f\|
    \qquad (f\in L^2_0(S,m)).
\]
Recall that $\mu$ is also assumed to have a finite exponential moment.
By Lemma~\ref{L:conv-of-exp} the same applies to convolution powers
and their convex combinations. 
Hence there exists $\epsilon>0$ so that
\[
    M=\sum_{\gamma\in\Gamma} \bar\mu(\gamma)\cdot e^{\epsilon\cdot |\gamma|_A}
\]
is finite. Hence
\[
    \sum_{\gamma\in\Gamma} \mu'(\gamma)\cdot e^{\epsilon\cdot |\gamma|_A}
    \le \frac{M}{1-p}<\infty.
\]
Recall the bounded operators $T_g$ of $L^2(S,m)$ as in \S~\ref{subs:reps}.
By definition of $\alpha_\rho$ we have $\|T_g\|\le e^{\alpha_\rho}$ for all $g\in \rho(A)$.
Since $\|T_{g_1g_2}\|\le \|T_{g_1}\|\cdot \|T_{g_2}\|$ it follows 
\[
    \|T_{\rho(\gamma)}\|\le e^{\alpha_\rho\cdot|\gamma|_A}\qquad (\gamma\in\Gamma).
\]
Therefore we can estimate the norm $\|Q^*\|=\|Q\|$ by
\[
    \begin{split}
        \|Q\| &\le \sum_{\gamma} \mu'(\gamma)\cdot\|T_{\rho(\gamma)}\|
            \le \sum_{\gamma} \mu'(\gamma)\cdot e^{\alpha_\rho\cdot|\gamma|_A}\\
            &=1+ \sum_{\gamma\ne e} \mu'(\gamma)\cdot \left(e^{\alpha_\rho\cdot |\gamma|_A}-1\right)\\
            &=1+\sum_{\gamma\ne e} \mu'(\gamma)\cdot e^{\epsilon\cdot|\gamma|_A}\cdot \frac{e^{\alpha_\rho\cdot |\gamma|_A}-1}{e^{\epsilon\cdot|\gamma|_A}}\\
            &\le 1+\frac{2M}{\epsilon(1-p)}\cdot \alpha_\rho
    \end{split}
\]
using Lemma~\ref{L:elementary} in the last step and
assuming $\alpha_\rho<\epsilon/2$.
Taking 
\[
    \Csg:=\frac{\epsilon p}{2M}>0
\]
we deduce that the assumption
$\alpha_\rho<c_1\cdot\kappa_B(\rho)$ guarantees that
\[
    \|Q^*\|<1+\kappa_{\rho(B)}\frac{p}{1-p}
\]
and Lemma~\ref{L:regularity} applies.
Thus there is $\phi\in L^2(S,m)\cap L_+^2(X,m)$ with 
\[
    \int_S \phi\dd m=1,\qquad T^*\phi=\phi.
\]
Let $\bar\nu$ denote the probability measure $\dd \bar\nu=\phi\dd m$.
(In this proof we drop the subscript $\rho$ from the notation for the sake of readability).
We note that $\bar\nu$ is $\bar\mu$-stationary because
\begin{equation}\label{e:phi4barmu}
    \phi(x)=T^*\phi(x)=\sum_{\gamma} \bar\mu(\gamma)\cdot \frac{\dd \rho(\gamma)_* m}{\dd m}(x)
    \cdot \phi(\rho(\gamma). x).
\end{equation}
Since $\nu=\mu*\nu=\mu*\mu*\nu=\dots$ the measure $\nu$ is stationary for all convolution powers of $\mu$ and their convex combinations, such as $\bar\mu$. Hence we have
\[
    \bar\mu*\bar\nu=\bar\nu,\qquad \bar\mu*\nu=\nu.
\]
But $\supp(\bar\mu)\supset\supp(\mu)$ generates $\Gamma$, 
thus by Theorem~\ref{T:Zd-bnd} there is only one
$\bar\mu$-stationary measure. Hence $\nu=\bar\nu$.

Since $\phi\in L^2(S,m)$ by construction, it remains to show that $\phi^{-\tau}\in L^1(S,m)$ 
for some $\tau>0$. 
To this end, notice that equation (\ref{e:phi4barmu}) implies that 
\[
    \begin{split}
    \phi(x) &\ge \frac{p}{2k}\cdot \sum_{g\in \rho(B)\cup \rho(B)^{-1}} 
    \frac{\dd g_* m}{\dd m}(x)\cdot \phi(g.x)\\
    &\ge r\cdot \max_{g\in \rho(B)\cup \rho(B)^{-1}} \phi(g.x)
    \end{split}
\]
where $r>0$ takes into account both $t/2k$ but also 
the essential minimum of the Radon--Nikodym
derivatives of the finitely many elements $g$ participating in the sum.

We shall now use the spectral gap assumption $\kappa_{\rho(B)}>0$ on $L^2_0(S,m)$.
Fix some $\delta>0$ so that the set $E=\setdef{x\in S}{\phi(x)>\delta}$ has positive measure $m(E)>0$. 
Denote by $E_0=E$ and let $E_n$ denote the set $x\in X$ for which
there exists $g\in \rho(B)\cup \rho(B)^{-1}$ so that $gx\in E_{n-1}$.
Note that the inequality 
\[
    \phi(x)\ge r\cdot \max_{g\in \rho(B)\cup \rho(B)^{-1}} \phi(g.x)
\]
implies that for a.e. $x\in E_n$ we have $\phi(x)\ge \delta\cdot r^n$.
The spectral gap assumption $\kappa_{\rho(B)}>0$ implies that $m(S\setminus \bigcup E_n)=0$
and moreover there exists $\lambda<1$ and $C$ so that $m(E_n)<C\lambda^n$
(see Zimmer \cite{Zimmer}, or \cite{FM}*{Lemma 4.5}).
Taking $\tau>0$ so that $\lambda<r^\tau$, and thereby $\lambda\cdot r^{-\tau}<1$, we deduce
\[
    \int_S \phi^{-\tau}\dd m\le \sum_{n=0}^\infty 
    C\lambda^n(\delta r^n)^{-\tau}=C\delta^{-\tau}\sum_{n=0}^\infty 
    (\lambda\cdot r^{-\tau})^n<\infty.
\]
This proves the claimed integrability bounds.
This completes the proof of Theorem~\ref{T:Lebesgue-stat}.
\end{proof}

%%%%%%%%%%%%%%%%%%%%%%%%%%%%%%%%%%%%%%%%%%%%%%%%%%%%%%%%%%%%%%%%%%%%%%%%%%%%%%%%%%%%%%%%%%%%
\section{Continuity of Entropy}

Recall notation~\ref{N:Zdsg} for the set of representations $\rho:\Gamma\overto{} G$
satisfying Theorem~\ref{T:Lebesgue-stat} 
\[
    \HomZdsg(\Gamma,G;\mu,A,B)\subset \HomZd(\Gamma,G).
\]
The goal of this section is to prove that the map
\[
    \HomZdsg(\Gamma,G;\mu,A,B)\ \overto{}\ \Bnd(\Gamma,\mu)\ \overto{h_\mu}\ \bbR_+,
    \qquad \rho\ \mapsto\ h_\mu(B_\rho,\nu_\rho),
\]
associating to a representation $\rho$ the entropy of the corresponding boundary 
$(B_\rho,\nu_\rho)$, is continuous. 
This continuity, described in Theorem~\ref{T:continuity} below, 
can be shown for a slightly
broader family 
$\HomZdac(\Gamma,G;\mu)\supset \HomZdsg(\Gamma,G;\mu,A,B)$ 
of representations
consisting of all representations with image $\rho(\Gamma)$ 
being Zariski dense in $G$,
and such that the unique $\mu$-stationary measure $\nu_\rho$ on $S=G/P$ 
is in the $G$-invariant measure class $[m]$, 
where $\phi=\dd\nu_\rho/\dd m$ and $f=\log(\dd\nu_\rho/\dd m)$
are in $L^2(S,m)$.
We note that Theorem~\ref{T:Lebesgue-stat} establishes the inclusion: 
\[
   \HomZdsg(\Gamma,G;\mu,A,B)\subset \HomZdac(\Gamma,G;\mu)
\]
and as a consequence $h_\mu(\nu_\rho)$ varies continuously
over $\HomZdsg(\Gamma,G;\mu,A,B)$.

\begin{theorem}[Continuity of entropy]\label{T:continuity}\hfill{}\\
    Let $\Gamma$ be a finitely generated group, $\mu$ a generating measure with finite first moment, $G$ a semi-simple real Lie group. 
    Then the entropy map 
    \[
        \rho\in \HomZdac(\Gamma,G;\mu) \overto{} \bbR_+,\qquad \rho\mapsto h_\mu(B_\rho,\nu_\rho)
    \]
    is continuous with respect to the topology inherited from $\Hom(\Gamma,G)$.
\end{theorem}

\medskip

The first observation concerns weak-* convergence of stationary measures in $\Prob(S)$, $S=G/P$.
More precisely 
\begin{lemma}\label{L:stationary_limit}
    Let $\rho_n\in \Hom(\Gamma,G)$, $n\in\bbN$, be a sequence
    converging to a limit representation $\rho\in \HomZd(\Gamma,G)$.
    For $n\in\bbN$ let $\nu_n\in\Prob(S)$ be a $\rho_n(\mu)$-stationary measure, i.e.
    a measure satisfying
    \[
        \nu_n=\sum_{\gamma\in\Gamma} \rho_n(\gamma)_*\nu_n.
    \]
    Then $\nu_n$ weak-* converges to $\nu_\rho\in\Prob(S)$
    -- the unique $\rho(\mu)$-stationary measure.
\end{lemma}
\begin{proof}
    Since $\Prob(S)$ is weak-* compact, it suffices to show that:
    if a subsequence $\{\nu_{n_i}\}$ weak-* converges to some $\nu\in\Prob(S)$,
    then necessarily $\nu=\nu_\rho$.
    
    Let $f\in C(S)$ be a continuous function. For any $\gamma\in\Gamma$ we have
    \[
        \int_S f\circ \rho(\gamma)\dd\nu
        =\lim_{i\to\infty}\int_S f\circ\rho_{n_i}(\gamma)\dd\nu_{n_i} 
    \]
    where all the terms are bounded by $\|f\|$.
    Taking the average over $\gamma\in \Gamma$ with distribution $\mu$,
    and using $\rho_{n_i}(\mu)$-stationarity of $\nu_{n_i}$,
    we get
    \[
        \begin{split}
            \int_\Gamma \int_S f\circ\rho(\gamma)\dd\nu\dd\mu(\gamma)
        &=\lim_{i\to\infty}\int_\Gamma \int_S 
        f\circ\rho_{n_i}(\gamma)\dd\nu_{n_i}\dd\mu(\gamma)\\
        &=\lim_{i\to\infty} \int_S f\dd \nu_{n_i}=\int_S f\dd \nu.
        \end{split}
    \]
    Hence $\nu$ is $\rho(\mu)$-stationary. 
    Since $\rho\in \HomZd(\Gamma,G)$, there is only one $\rho(\mu)$-stationary measure
    $\nu_\rho$. Therefore $\nu=\nu_\rho$.
\end{proof}

Hence the map $\HomZd(\Gamma,G)\overto{} \Prob(S)$, $\rho\mapsto \nu_\rho$, is weak-* continuous.
It is not clear how this fact can be used towards continuity of the entropy
\begin{equation}\label{e:entropy-rho}
    h_\mu(\nu_\rho)=\int_\Gamma \int_S -\log\frac{\dd \rho(\gamma)\iv_*\nu_\rho}{\dd\nu_\rho}(\xi)\dd\nu_\rho(\xi)\dd\mu(\gamma) 
\end{equation}
due to possibly irregular dependence on $\rho$ of the integrand.

To overcome this difficulty we make the following observations.
There exists an everywhere defined \textbf{continuous} cocycle
\[
    \sigma_0:G\times S\overto{} \bbR
\]
that represents the $m$-a.e. defined Radon--Nikodym cocycle 
\[
    \sigma_0(g,\xi)=-\log \frac{\dd g\iv_*m}{\dd m}(\xi).
\]
For $\rho\in \Hom(\Gamma,G)$ we have $|\sigma_0(\rho(a),\xi)|\le\alpha_\rho$
for $m$-a.e. $\xi\in S$.
Therefore 
\[
    \|\sigma_0(\rho(\gamma),-)\|_\infty
    \le \alpha_\rho\cdot |\gamma|_A.
\]
Thus we have the estimate
\begin{equation}\label{e:sigma0-bound}
    \int_\Gamma\int_S  |\sigma_0(\rho(\gamma),\xi)|\dd\nu_\rho(\xi)\dd\mu(\gamma)\le
    \alpha_\rho\cdot\left(\sum_{\gamma\in\Gamma} \mu(\gamma)\cdot |\gamma|_A\right)<\infty.
\end{equation}
The following general lemma is well known 
(cf. \cite{FK}), but for the reader's convenience
we include here its simple proof.
We consider an abstract setting that involves
a $\Gamma$-space $M$ with a $\mu$-stationary measure $\nu$.

\begin{lemma}[Cohomology lemma] \label{L:cohomology}\hfill{}\\
    Let $(M,\nu)$ be a $(\Gamma,\mu)$-stationary space,
    and let $\eta$ be a measure on $M$ in the measure class 
    of $\nu$ with $\log\frac{\dd\nu}{\dd\eta}\in L^1(\nu)$.
    Denote by $\sigma_\nu,\sigma_\eta:\Gamma\times M\overto{}\bbR$ the $\nu$-a.e. defined measurable cocycles
    \[
        \sigma_\nu(g,\xi)=-\log\frac{\dd g_*^{-1}\nu}{\dd\nu}(\xi),
        \qquad
        \sigma_\eta(g,\xi)=-\log\frac{\dd g_*^{-1}\eta}{\dd\eta}(\xi).
    \]
    and assume $\sigma_\eta\in L^1(\Gamma\times M,\mu\times\nu)$.
    Then $\sigma_\nu\in L^1(\Gamma\times M,\mu\times\nu)$ and
    \[
        \int_\Gamma\int_M\sigma_\nu(g.\xi)\dd\nu(\xi)\dd\mu(g)
        =\int_\Gamma\int_M\sigma_\eta(g.\xi)\dd\nu(\xi)\dd\mu(g).
    \]
\end{lemma}
\begin{proof}
     For every $g\in \Gamma$ we have an a.e. identity  
    \[ 
        \begin{split}
        \frac{\dd g\iv_*\eta}{\dd\eta}(\xi) 
        &= \frac{\dd g\iv_*\eta}{\dd g\iv_* \nu}(\xi) \cdot \frac{\dd g\iv_*\nu}{\dd \nu}(\xi) 
        \cdot \frac{\dd \nu}{\dd \eta}(\xi)\\
        &=\frac{\dd \eta}{\dd \nu}(g.\xi)
        \cdot \frac{\dd g\iv_*\nu}{\dd \nu}(\xi) 
        \cdot \frac{\dd \nu}{\dd \eta}(\xi)
        \end{split}
    \]
    and therefore 
    \[
        \sigma_\eta(g,\xi)=\sigma_\nu(g,\xi)+
        f(g.\xi)-f(\xi).
    \]
    where $f(\xi)=\log (\dd \nu/\dd \eta)(\xi)$.
    We observe that
    \[
        \int_\Gamma\int_M (f\circ g-f)\dd\nu\dd\mu
        =\int_M f\dd \mu*\nu-\int_M f\dd\nu=0.
    \]
    This proves the Lemma.
\end{proof}

\begin{proof}[Proof of Theorem~\ref{T:continuity}]
    Consider $\rho\in \HomZdac(\Gamma,G)$,
    let $\nu_\rho$ be the associated 
    stationary measure on $S$, and let $\phi=\dd \nu_\rho/\dd m$
    and $f=\log \phi$.
    By our assumption $\phi,f\in L^2(S,m)$, and so Cauchy--Schwarz
    inequality implies that $f\in L^1(S,\nu_\rho)$ because
    \[
        \int_S |f|\dd \nu_\rho=\int_S |f|\cdot \phi\dd m<\infty.
    \]
    In view of (\ref{e:sigma0-bound}) it follows that
    Lemma~\ref{L:cohomology} with $\nu=\nu_\rho$ and $\eta=m$
    applies, and so
    \[
        h_\mu(\nu_\rho)=\int_\Gamma\int_S 
        \sigma_0(\rho(\gamma),\xi)\dd\nu_\rho(\xi)\dd\mu(\gamma).
    \]
    Let $\rho_n, \rho$ be a representations in $\HomZdac(\Gamma,G;\mu)$ so that $\rho_n\to \rho$
    in $\Hom(\Gamma,G)$.
    Let $\nu_{\rho_n}$, $\nu_\rho$ be the associated stationary measures on $S$.
    By Lemma~\ref{L:stationary_limit} we have weak-* convergence $\nu_{\rho_n}\to \nu_\rho$.
    
    Since $\sigma_0:G\times S\overto{}\bbR$ is a continuous map, 
    for every $\gamma\in\Gamma$ 
    we have convergence as $n\to\infty$
    \[
        \int_S \sigma_0(\rho_n(\gamma),\xi)\dd\nu_{\rho_n}(\xi)
        \quad\overto{}\quad
        \int_S \sigma_0(\rho(\gamma),\xi)\dd\nu_{\rho}(\xi).
    \]
    Since $\sup_n \alpha_{\rho_n}<\infty$ estimate (\ref{e:alpha-norm}) justifies
    (dominated convergence) the limit
    \[
        \int_\Gamma\int_S \sigma_0(\rho_n(\gamma),\xi)\dd\nu_{\rho_n}(\xi)\dd\mu(\gamma)
        \quad\overto{}\quad
        \int_\Gamma\int_S \sigma_0(\rho(\gamma),\xi)\dd\nu_{\rho}(\xi)\dd\mu(\gamma).
    \]
    In view of Lemma~\ref{L:cohomology} we deduce
    \[
        h_\mu(B_{\rho_n},\nu_{\rho_n}) \quad\overto{}\quad h_\mu(B_{\rho},\nu_{\rho})
    \]
    that proves the Theorem.
\end{proof}

\section{Estimating Entropy range}

So far we established that every $\rho\in\HomZdsg(\Gamma,G;\mu,A,B)$ gives rise to a 
boundary $(B_\rho,\nu_\rho)$ which is non-trivial, and therefore has
a strictly positive entropy $h_\mu(B_\rho,\nu_\rho)>0$.
The goal of this section is to give an estimate for this entropy
in terms of the quantity
\[
    \alpha_\rho=\max_{\gamma\in A\cup A^{-1}}\log\NRN(\rho(\gamma))
\]
that gauges the non-triviality of the boundary $(B_\rho,\nu_\rho)$. 
We focus on $G=\SO(3,1)$, $K=\SO(3)$, and $S=S^2$.

\begin{theorem}[Entropy estimate]\label{T:lower_bound}\hfill{}\\
    Let $\Gamma$ be a finitely generated group, 
    $\mu$ be a probability measure on $\Gamma$
    with finite exponential moment, so that $\supp(\mu)$ generates $\Gamma$ as a semi-group.
    Let $G=\SO(3,1)$, $K=\SO(3)$, $S=S^2$ and $B\subset \Gamma$ be a finite subset 
    with $\rho_0(B)\subset K$ having a spectral gap $\kappa_{\rho(B)}>0$.
    Then there is $\Cent>0$, $L$ and $\alpha_0>0$ so that 
    \[
        \Cent\cdot (\alpha_\rho)^2
        \le  h_\mu(B_\rho,\nu_\rho)\le L\cdot\alpha_\rho.
    \]
    for all $\rho\in \HomZdsg(\Gamma,G;\mu,A,B)$ with $0<\alpha_\rho\le \alpha_0$.
\end{theorem}

The upper estimate follows from Lemma~\ref{L:cohomology} and formula 
(\ref{e:sigma0-bound}) where
\[
    L=\sum_\gamma \mu(\gamma)\cdot|\gamma|_A.
\]
The challenge in establishing the lower estimate is that the definition 
(\ref{e:entropy-rho}) of $h_\mu(B_\rho,\nu_\rho)$ (or Lemma~\ref{L:cohomology} 
for that matter) involves measure $\nu_\rho$ which is not known explicitly.
To overcome this difficulty we use the quasi-regular 
unitary representation
\[
    \pi_\rho:\Gamma\overto{} U(L^2(S,\nu_\rho))
\]
given by
\[
    (\pi_\rho(\gamma)f)(x)
    =\left(\frac{\dd \rho(\gamma)^{-1}_*\nu_\rho}{\dd\nu_\rho}(x)\right)^{1/2}\cdot f(\rho(\gamma)^{-1}x).
\]
We denote by $\pi_\rho(\mu)$ the contraction on $L^2(S,\nu_\rho)$ given by the average of unitary operators
\[
    \pi_\rho(\mu)=\int_\Gamma \pi_\rho(\gamma)\dd\mu(\gamma),
\]
and by $|\pi_\rho(\mu)|_{\rm sp}\le 1$ the spectral radius of this contraction.

\begin{lemma}\label{L:ent2spec}
    With the notation above we have the estimate
    \[
        h_\mu(S,\nu_\rho)%\ge -2\log \ip{\pi_\rho(\mu)\mathbb{1}}{\mathbb{1}}
        \ge -2\log |\pi_\rho(\mu)|_{\rm sp}
    \]
    in terms of the spectral radius.
\end{lemma}
\begin{proof}
    Using convexity of $f(t)=-\log(t)$ we get
    \[
        \begin{split}
            h_\mu(S,\nu_\rho)&=\int_\Gamma\int_S
            -\log\frac{\dd\rho(\gamma)\iv_*\nu_\rho}{\dd\nu_\rho}\dd\nu_\rho\dd\mu\\
            &=2\cdot\int_\Gamma\int_S
            -\log\left(\frac{\dd\rho(\gamma)\iv_*\nu_\rho}{\dd\nu_\rho}\right)^\half
            \dd\nu_\rho\dd\mu\\
            &\ge -2\cdot\log \left(\int_\Gamma\int_S 
            \left(\frac{\dd\rho(\gamma)\iv_*\nu_\rho}{\dd\nu_\rho}\right)^\half
            \dd\nu_\rho\dd\mu
            \right)\\
            &=-2\cdot \log\ip{\pi_\rho(\mu)\mathbb{1}}{\mathbb{1}}
            \ge -2\log \|\pi_\rho(\mu)\|.
        \end{split}
    \]
    Applying this argument to convolution powers $\mu^{*n}$ we get
    \[
        h_\mu(S,\nu_\rho)=\frac{1}{n} h_{\mu^{*n}}(S,\nu_\rho)
        \ge -2\frac{1}{n}\log \|\pi_\rho(\mu^{*n})\|=-2\log\|\pi_\rho(\mu)^n\|^{1/n}.
    \]
    Passing to the limit as $n\to\infty$ we deduce the claimed inequality with spectral radius.
\end{proof}

Our next observation is that since $\nu_\rho$ is in the same measure class as $m$,
there is a unitary isomorphism $L^2(S,\nu_\rho)\overto{\cong}L^2(S,m)$
% \[
%     L^2(S,\nu_\rho) \quad \overto{\cong}\quad L^2(S,m),
%     \qquad 
%     f(x)\ \mapsto\ \left(\frac{\dd m}{\dd\nu_\rho}(x)\right)^\half\cdot f(x) 
% \]
that intertwines $\pi_\rho$ with the unitary representation $\pi\circ\rho$ 
from (\ref{e:quasi-reg-m}), and so $\pi_\rho(\mu)$ on $L^2(S,\nu_\rho)$
is isomorphic to the operator $P_\mu=\pi\circ\rho(\mu)$ on $L^2(S,m)$
\[
    (P_\mu f)(\xi)=\int_\Gamma \left(\frac{\dd \rho(\gamma)_*m}{\dd m}(\xi)\right)^\half\cdot 
    f\left(\rho(\gamma)\iv \xi\right)\dd\mu(\gamma).
\]
% \[
%     (P_\mu f)(x)=
%     \int_\Gamma \left(\frac{\dd \rho(\gamma)^{-1}_*m}{\dd m}(x)\right)^{1/2} f(\rho(\gamma)^{-1}x)\dd\mu(\gamma)
% \]
As in the proof of Theorem~\ref{T:Lebesgue-stat}, we choose a convex combination $\bar\mu$ 
\begin{equation}\label{e:barmu}
    \bar\mu=\sum_{i=1}^{k} p_i\cdot\mu^{*n_i}
\end{equation}
of convolution powers of $\mu$, so that $\bar\mu\ge p\cdot \mu_0$ 
and $\bar\mu(a)\ge q$ for all $a\in A$, where $\mu_0$ is the uniform distribution
on $B\cup B^{-1}$, and $p>0$ and $q>0$ are some constants. 
Since $\pi\circ \rho$ preserve the cone of positive functions,
the spectral radii of $P_\mu$ and $P_{\bar\mu}$ are attained by positive functions,
and one can see that
\[
    \srad{P_{\bar\mu}}\ge F\left(\srad{P_\mu}\right),
    \qquad \textrm{where}\qquad
    F(t)=\sum_{i=1}^k p_i\cdot t^{n_i}.
\]
The function $F$ is strictly convex with $F(1)=1$, hence
\[
    F\left(\srad{P_\mu}\right)\ge 1-F'(1)\cdot (1-\srad{P_\mu})
\]
where the constant $F'(1)=\sum_{i=1}^k n_ip_i$ depends only on the fixed data 
$(\Gamma,\mu)$ and $A,B\subset \Gamma$.
In view of Lemma~\ref{L:ent2spec}, and using the elementary inequality $-\log(1-x)>x$ for $x\in (0,1)$, 
we have
\begin{equation}\label{e:ent-srad}
       \begin{split}
     h_\mu(\nu_\rho)&\ge -2\cdot\log\left(\srad{P_\mu}\right)
     \ge 2\left(1-\srad{P_\mu}\right)\\
     &\ge 
     \frac{2}{F'(1)}\cdot \left(1-\srad{P_{\bar\mu}}\right)
     \ge \frac{2}{F'(1)}\cdot (1-\|P_{\bar\mu}\|).
    \end{split}
\end{equation}
It remains to get a linear lower estimate on the spectral gap 
$1-\|P_{\bar \mu}\|$ on $L^2(S,m)$ in terms of the quantity
\[
    \alpha_\rho=\max_{a\in A}\ \log\NRN(\rho(a)).
\]
We have $\bar\mu=p\cdot \mu_0+(1-p)\cdot \mu'$ for some $0<p<1$,
where $\mu_0$ is a uniform distribution on $B\cup B^{-1}$, and $\mu'$ is some probability 
measure on $\Gamma$, where $\mu'(a)\ge q/(1-p)\ge q$ fpr all $a\in A$.
We have 
\[
    P_{\bar\mu}=p\cdot P_{\mu_0}+(1-p)\cdot P_{\mu'}.
\]
Denote $\epsilon=1-\|P_{\bar\mu}\|$ and let $f\in L^2(S,m)$ be a unit vector
with $\|P_{\bar\mu}f\|\ge 1-2\epsilon$. We may assume $f$ to be positive $f\ge 0$.
Since  
\[
    1-2\epsilon\le \|P_{\bar\mu}f\|\le p\cdot\|P_{\mu_0}f\|+(1-p)\cdot\|P_{\mu'}f\|,
\]
while $\|P_{\mu_0}f\|\le 1$ and $\|P_{\mu'}f\|\le 1$, it follows that
\begin{equation}\label{e:2ineq}
    1-\frac{2\epsilon}{p}\le \|P_{\mu_0}f\|
    \qquad\textrm{and}\qquad
    1-\frac{2\epsilon}{1-p}\le \|P_{\mu'}f\|.
\end{equation}
Write $f=f_0+(1-\delta)\cdot\mathbb{1}$ with $f_0\in L^2_0(S,m)$.
\begin{claim}
    $\delta<25(p\kappa_0)^{-2}\cdot\epsilon^2<C_1 \epsilon$.
\end{claim}
\begin{proof}
    We have $\|f_0\|=\sqrt{1-(1-\delta)^2}=\sqrt{2\delta-\delta^2}\le \sqrt{2\delta}$ 
    and so $\|P_{\mu_0}f_0\|\le (1-\kappa_0)\|f_0\|\le (1-\kappa_0)\sqrt{2\delta}$.
    Hence by triangle inequality
    \[
        \begin{split}
         1-\frac{2\epsilon}{p}&\le \|P_{\mu_0}f_0\|+(1-\delta)
        \le (1-\kappa_0)\sqrt{2\delta}+(1-\delta)\\
        &\le 1-(\sqrt{2}-1)\kappa_0\sqrt{\delta}
        \end{split}
    \]
    using $\delta<\sqrt{\delta}$. Thus
    $
    \sqrt{\delta}<2\epsilon/(\sqrt{2}-1)p\kappa_0<(5/p\kappa_0)\cdot \epsilon.
    $
\end{proof}
Using the second inequality (\ref{e:2ineq}), for any $a\in A$ we have
\[
    \begin{split}
        1-\frac{2\epsilon}{1-p}&\le \ip{P_{\mu'}f}{f}
        \le \ip{P_{\mu'}\mathbb{1}}{\mathbb{1}}+2\delta\\
        &\le \mu'(a)\cdot\ip{\pi(\rho(a))\mathbb{1}}{\mathbb{1}} +(1-\mu'(a))
        +2\delta\\
        &=1-\mu'(a)\cdot (1-\ip{\pi(\rho(a))\mathbb{1}}{\mathbb{1}})
       +2\delta.
    \end{split}
\]
This gives
$
    \epsilon+(1-p)\delta>\frac{q}{2}\cdot
    \left(1-\ip{\pi(\rho(a))\mathbb{1}}{\mathbb{1}}\right)
$, and in particular
\begin{equation}\label{e:moving1}
    1-\|P_{\bar\mu}\|=\epsilon>c_2\cdot \left(1-\ip{\pi(\rho(a))\mathbb{1}}{\mathbb{1}}\right)
\end{equation}
for some constant $c_2$ that depends only on $\Gamma$, $\mu$, $\kappa_0$.

\bigskip

We now specialize to $G=\SO(3,1)$ and $K=\SO(3)$
or, equivalently, $G=\PSL_2(\bbC)$ and $K=\SU(2)$.

\begin{lemma}\label{L:matrix-coef}
    Let $g\in G$ have $\log\NRN(g)=t$.
    Then 
    \[
        \ip{\pi(g)\mathbb{1}}{\mathbb{1}}=\frac{t}{2\sinh(t/2)}\le 1-\frac{t^2}{25}
    \]
    for $0<t\le 1$.
\end{lemma}
\begin{proof}
    Since the constant function $\mathbb{1}$ is invariant under $K$,
    the value of $\ip{\pi(g)\mathbb{1}}{\mathbb{1}}$ is bi-$K$-invariant
    and can be reduced to a diagonal $g$.
    Viewing $G$ as $\PSL_2(\bbC)$ acting on the Riemann sphere $\hat\bbC=\bbC\cup\{\infty\}$
    we may assume $g$ to be the map $z\mapsto e^{t/2}z$ of $\bbC$
    for some $t\ge0$.
    Under this identification the $K$-invariant probability measure 
    $m$ has density 
    \[
        \frac{2r}{(1+r^2)^2}\dd r\dd \theta
    \]
    in polar coordinates. The Radon--Nikodym derivative is then
    \[
        \frac{\dd g\iv_* m}{\dd m}(r,\theta)=\frac{e^t\cdot (1+r^2)^2}{(1+e^t r^2)^2}.
    \]
    Thus $\NRN(g)=e^t$, while  
    \[
        \begin{split}
        \ip{\pi(g)\mathbb{1}}{\mathbb{1}}&=
         \int_0^\infty \frac{e^{t/2}\cdot (1+r^2)}{(1+e^t r^2)}
        \cdot\frac{2 r\dd r}{(1+r^2)^2}=\frac{t/2}{\sinh(t/2)}\\
        &\le \frac{1}{1+t^2/24}= 1-\frac{t^2}{24}+\left(\frac{t^2}{24}\right)^2-\dots
        \le 1-\frac{t^2}{25}.
        \end{split}
        %=\frac{1}{1+t^2/3!+t^4/5!+}.
    \]
\end{proof} 

\medskip
Finally, combining estimates (\ref{e:ent-srad}), (\ref{e:moving1}) 
and Lemma~\ref{L:matrix-coef} we get
\[
    h_\mu(\nu_\rho)\ge \frac{2}{F'(1)}(1-\|P_{\bar\mu}\|)
    \ge \frac{2 c_2}{F'(1)}\cdot (\alpha_\rho)^2
\]
completing the proof of Theorem~\ref{T:lower_bound}.

\section{Proof of Theorem~\ref{T:main-realization}}

\subsection{Minimality and isomorphism of constructed boundaries}

In the proof of Theorem~\ref{T:main-realization} we shall use the following 
two Propositions in the case of $G=\SO(3,1)$ and $G/H=S^2$.
In the proof of Theorem~\ref{T:main-cubes} we shall use
the products $G=\SO(3,1)^d$ and $G/H=(S^2)^d$ -- product of $d$ copies of $2$-spheres.

\begin{prop}\label{P:qutients}
    Let $G$ be a lcsc group, $H<G$ a closed subgroup, and
    $\rho:\Gamma\overto{}G$ a homomorphism with a dense image.
    Let $\nu$ be a probability measure on $G/H$ in a $G$-invariant measure class.
    
    Then the only measurable $\Gamma$-equivariant quotients of $(G/H,\nu)$
    are $G$-equivariant, i.e. have the form $G/H\overto{} G/L$, $gH\mapsto gL$
    where $H<L<G$ is an intermediate closed subgroup.
\end{prop}
% \begin{proof}
%     \markred{Write up proof}
% \end{proof}

\begin{prop}\label{P:isoms}
    Let $G$ be a lcsc group, $H<G$ be a closed subgroup, and
    $\rho_1, \rho_2:\Gamma\overto{}G$ be two homomorphisms with a dense images.
    Then the only measurable measure-class preserving map $\phi:G/H\overto{} G/H$ 
    satisfying a.e.
    \[
        \phi\circ \rho_1(\gamma)=\rho_2(\gamma)\circ \phi
    \]
\end{prop}
% \begin{proof}
%     \markred{Write up proof}
% \end{proof}

\medskip

\subsection{Continuous deformation in $\HomZdsg(\Gamma,G)$}

\subsection*{The case of Free Groups}

Let $F_n=\langle a_1,\dots,a_n\rangle$ be a free group on $n\ge 3$
free generators. Denote $A=\{a_1,\dots,a_n\}$ and $B=\{a_2,\dots,a_n\}$.
Let $\Csg>0$ be the constant from Theorem~\ref{T:Lebesgue-stat}, 
that depends only on $F_n$ and $\mu$.

Fix a subset $\{W_2,\dots,W_n\}\subset \SO(3)$ 
with a spectral gap 
$\kappa_0>0$, for example using Theorem~\ref{T:LPS}.
Consider a continuous family of distinct elements $T_{s,r}\in G$ parametrized by 
$s\in [0,1]$, $r\in (0,\infty)$, where
\[
    \log\NRN(T_{s,r})=r.
\]
In view of Lemma~\ref{L:matrix-coef}, the maps
$T_{s,r}:z\mapsto e^{r/2+i s}z$ of the Riemann sphere
provide an example of such a family.
Let $\rho_{s,r}\in \Hom(F_n,G)$  be the homomorphism defined by
\[
    \rho_{s,r}(a_1)=T_{s,r},\qquad \rho(a_2)=W_2,\quad \dots\quad \rho(a_n)=W_n.
\]
Note that $\alpha_{\rho_{s,r}}=r$ because $a_1$ is the only element of $A$
that is mapped by $\rho_{s,r}$ outside of $K$, while $\rho_{s,r}(a_1)=T_{s,r}$
with $\log\NRN(T_{s,r})=r$.

Theorem~\ref{T:Lebesgue-stat} determines a positive constant
$\Csg>0$ that guarantees that for $s\in [0,1]$ and $0<r<\Csg\cdot \kappa_0$
the representation $\rho_{s,r}$ is in $\HomZdsg(F_n,G;\mu,A,B)$.
In particular, the $\rho_{s,r}(\mu)$-stationary measure $\nu_{\rho_{s,r}}$
gives a $\mu$-boundary. 
By Theorem~\ref{T:lower_bound} its entropy 
satisfies 
\[
    \Cent\cdot r^2\le h_\mu(\nu_{\rho_{s,r}})\le L \cdot r.
\]
Since $\rho_{s,r}\in \Hom(F_n,G)$ varies continuously,
Theorem~\ref{T:continuity} guarantees that for each $s\in [0,1]$ the value  
of $h_\mu(\nu_{\rho_{s,r}})$ varies continuously in $r$.
Let 
\[
    h_1:=\Cent\cdot (\Csg \kappa_0)^2.
\]
The Intermediate Value Theorem implies that 
for every $s\in [0,1]$ for each $t\in (0,h_1)$ 
there is $0<r(t)<\Cent(\Csg \kappa_0)^2$ so that 
the representation $\rho_{s,r(t)}:\Gamma\overto{} G$
defines a $\mu$-boundary $B_{s,t}$ with $h_\mu(B_{s,t})=t$.

Since $T_{s,r}$ are all distinct, the representations $\rho_{s,r}$
are all different.
Proposition~\ref{P:isoms} implies that the boundaries $B_{s,t}$ are non-isomorphic.
Furthermore, since $S=G/P$ has no non-trivial $G$-quotients,
Proposition~\ref{P:qutients} implies that $B_{s,t}$ have non-non-trivial measurable
$\Gamma$-quotients. 
We also note that the $\Gamma$-action on $G/P$ is essentially free for any atomless measure.
This completes the proof of Theorem~\ref{T:main-realization} for free groups.

\medskip

\subsection*{The case of Surface Groups}

Let $\Sigma_g$ denote a closed orientable surface of genus $g\ge 2$. 
Its fundamental group has a presentation
\begin{equation}\label{e:surface-amalgam}
    \Gamma=\pi_1(\Sigma_g)
    =\left\langle a_1,\dots,a_g,b_1,\dots,b_g 
    \mid [a_1,b_1]\cdots[a_g,b_g]=1\right\rangle
\end{equation}
and can be viewed as an amalgamated product $F_2*_\bbZ F_{2g-2}$
of free groups with free generators
\[
    F_2=\langle a_1,b_1\rangle,\qquad
    F_{2g-2}=\langle a_2,\dots,a_{g},b_2,\dots,b_g\rangle
\]
and $\bbZ=\langle c\rangle$ where $c$ embeds as
\[
    c=[a_1,b_1]^{-1}\in F_2,\qquad c=[a_2,b_2]\cdots[a_g,b_g]\in F_{2g-2}. 
\] 
Let $A=\{a_1,\dots,a_g,b_1,\dots,b_g\}$ be the standard generating set for $\Gamma$,
and $B=A\setminus \{a_1,b_1\}$.
Fix a representation $\rho_0:F_{2g-2}\overto{} \SO(3)$ 
so that $\rho_0(B)$ has a positive spectral gap $\kappa_0>0$, for example
using Theorem~\ref{T:LPS}.
Let $W_0\in\SO(3)$ denote the image of the commutator
\[
    W_0=\rho_0([a_2,b_2]\cdots[a_g,b_g]).
\]
Our fixed homomorphism $\rho_0:F_{2g-2}\overto{} G$ extends to 
a homomorphism $\rho:\Gamma\overto{} G$ iff the images
$T=\rho(a_1)$, $R=\rho(b_1)$ satisfy 
\begin{equation}\label{e:ghk}
    [T,R]\cdot W_0=1,\qquad\textrm{i.e.}\qquad TRT\iv R\iv W_0=1.
\end{equation}
It is well known that every element of $\SO(3)$ is a commutator,
so there exist $T_0,R_0\in\SO(3)$ satisfying relation (\ref{e:ghk}).

Consider a continuous family of distinct elements $V_{s,r}$ in the centralizer
of $R_0$ in $G=\PSL_2(\bbC)$,
parametrized by $s\in [0,1]$, $r\in (0,\infty)$, so that 
\[
    \log\NRN(V_{s,r})=r.
\]
Indeed, viewing $S=G/P$ as the Riemann sphere and assuming (up to conjugation) that $R_0$ fixes $0$ and $\infty$,
the family $V_{s,r}:z\mapsto e^{r/2+i s}z$ provides such an example 
(see Lemma~\ref{L:matrix-coef}).
Then $T_{s,r}=T_0V_{s,r}$ and $R_0$ still satisfy (\ref{e:ghk}) and we 
obtain a family of homomorphisms 
$\rho_{s,r}\in \Hom(\Gamma,G)$ by setting $\rho_{s,r}$ to be $\rho_0$ on $F_{2g-2}$, 
and setting 
\[
    \rho_{s,r}(a_1)=T_{s,r}=T_0V_{s,r},\qquad \rho_{s,r}(b_1)=R_0.
\]
Observe that
\[
    \alpha_{\rho_{s,r}}=\log\NRN(T_{s,r})=\log\NRN(V_{s,r})=r.
\]
Now the proof can be completed as in the case of the free groups.

\medskip

\section{Proof of Theorem~\ref{T:main-cubes}}

We first need the following result about spectral gaps in powers $K=\SU(2)^d$. 

\begin{theorem}[Spectral gap in products]\label{T:specgap-prods}\hfill{}\\
    Let $k\ge 2$ be fixed. Given $d\in\bbN$ there is
    a subset $\{V_1,\dots,V_k,U_1,\dots,U_k\}$ in the compact group 
    $K=\SU(2)^d$ with the spectral gap $\kappa_d>0$ on $L^2_0(K,m_K)$.\\
    In fact, we can assume $\kappa_d\ge ???$
\end{theorem}
We postpone the proof of this Theorem to the next section,
and now proceed with the construction of the cubes $C_{s,r}\in\Bnd(\Gamma,\mu)$.

\subsection{Proof of Theorem~\ref{T:main-cubes}}
Let $d\in\bbN$ be given. 
Consider product groups
\[
    G=\prod_{j=1}^d G_j,\qquad K=\prod_{j=1}^d K_j,
    \qquad G_j=\PSL_2(\bbC),\qquad K_j=\SU(2),
\]
and $S=\prod_{j=1}^d S_j$ -- product of spheres $S_j=\bbC\cup\{\infty\}$.
We denote by $m_K$ the Haar measure on $K=\SU(2)^d$, and by $m$
the $K$-invariant probability measure on $S$.
We denote by $\pr_j$ coordinate projections $G\overto{} G_j$, $K\overto{} K_j$,
$S\overto{} S_j$.

Recall that we consider free groups $\Gamma=F_n$, $n\ge 5$, with free generator $A=\{a_1,\dots,a_n\}$,
and surface groups $\Gamma=\pi_1(\Sigma_g)$, $g\ge 3$, with a standard generating set
$A=\{a_1,\dots,a_g,b_1,\dots,b_g\}$. 
We set $B=A\setminus\{a_1\}=\{a_2,\dots,a_n\}$ and 
$B=A\setminus\{a_1,b_1\}=\{a_2,\dots,a_g,b_2,\dots,b_g\}$ in these cases. 
Note that $B$ generate a free subgroup $\Gamma_0$ in $\Gamma$ with $B$ is a set of
free generators.
Note that the assumptions $n\ge 5$ and $g\ge 3$ ensure that 
the rank of $\Gamma_0$ is $|B|\ge 4$.

\medskip

\begin{lemma}\label{L:rho-s-r-properties}
    Let $\Gamma$ and $A,B\subset\Gamma$ be as above.
    Then there exist families $\{\rho_{s,\mathbf{r}}\in\Hom(\Gamma,G)\}$ indexed by 
    $s\in (0,1]$, $\mathbf{r}=(r_1,\dots,r_d)\in (0,\infty)^d$, so that for each $s$ and $\mathbf{r}$ we have
    \begin{enumerate}
    \item \label{i:BcalW} 
    The image $\calW=\rho_{s,\textbf{r}}(B)$ is contained in $K$ and has a spectral gap $\kappa_d>0$
    on $L^2_0(K,m_K)$.
    \item For each $s\in(0,1]$ the map $\textbf{r}\mapsto \rho_{s,\textbf{r}}$ is continuous.
    \item The projections $\rho_j=\pr_j\circ \rho_{s,\textbf{r}}:\Gamma\overto{} G_j$ 
    acting on the $2$-sphere $S_j$ satisfy \[\alpha_{\rho_j}=r_j\qquad (j=1,\dots, d).\]
    \item The representation $\rho_{s,\textbf{r}}$ acting on the $d$-product $S=\prod S_j$ satisfy \[\alpha_{\rho_{s,\textbf{r}}}=r_1+\dots+r_d.\]
    \item All representations $\pr_j\circ \rho_{s,\textbf{r}}:\Gamma\overto{}\PSL_2(\bbC)$
    are distinct.
    \item The image $\rho_{s,\textbf{r}}(\Gamma)$ is topologically dense in $G=\PSL_2(\bbC)^d$.
\end{enumerate}
\end{lemma}

\begin{proof}
Theorem~\ref{T:specgap-prods} provides a set $\calW$ of elements $\calW$
with cardinality $|\calW|=|B|$ in $K=\SU(2)^d$ with a spectral gap $\kappa_d>0$ on $L^2_0(K,m_K)$.
We fix a bijection between $B$ and $\calW$ and 
extend it to a representation
\[
    \rho_0:\Gamma_0\overto{} K
\]
and will construct a family of representations $\rho_{s,\textbf{r}}:\Gamma\overto{} G$ 
that extend $\rho_0:\Gamma_0\overto{} K\overto{} G$. 
This will establish property (\ref{i:BcalW}).

Consider the family of maps 
\[
    T_{s,\mathbf{r}}=(T^{(1)}_{s,r_1},\dots, T^{(d)}_{s,r_d})\in G=\PSL_2(\bbC)^d
\]
where the component maps $T^{(j)}_{s,r_j}\in \PSL_2(\bbC)$, viewed as acting on the Riemann sphere 
$\bbC\cup\{\infty\}$, are given by 
\begin{equation}\label{e:j-loxo}
     z\ \mapsto\ e^{r_j+\frac{s+j-1}{d}2\pi i}\cdot z.
\end{equation}
Then all these maps are distinct, and have
\[
    \log\NRN(T^{(j)}_{s,r_j})=r_j.
\]
In the case of $\Gamma=F_n$ we have $\Gamma_0\cong F_{n-1}$
and we can define $\rho_{s,\textbf{r}}\in \Hom(\Gamma,G)$ by setting
\[
    \rho_{s,\mathbf{r}}(a_1)=T_{s,\mathbf{r}},\qquad 
    \rho_{s,\mathbf{r}}(a_i)=\rho_0(a_i)\qquad (i=2,\dots,n).
\]
Then $\rho_{s,\textbf{r}}(B)=\rho_0(B)=\calW\subset K$ has
a spectral gap $\kappa_d>0$. 
Properties (2)-(5) follow directly from the construction.
To show property (6) consider the closure $H$ of the image $\rho_{s,\textbf{r}}(\Gamma)$
in $G=\PSL_2(\bbC)^d$.
Since $\rho_{s,\textbf{r}}(\Gamma_0)$ is dense in $K=\SU(2)^d$ we have $K<H$.
The projection of $H$ to each $\PSL_2(\bbC)$-factor is surjective because
$T_{s,r_j}^{(j)}$ does not lie in $\SU(2)$, but the latter is a maximal closed subgroup in $\PSL_2(\bbC)$.
It follows from Goursat's lemma that $H=G$, concluding the proof of Lemma~\ref{L:rho-s-r-properties}
for free groups.

\medskip

Next consider the case of a surface group 
$\Gamma=\pi_1(\Sigma_g)$ with $g\ge 3$, that we view as an amalgamated product 
$\Gamma=F_2*_\bbZ F_{2g-2}$ as in (\ref{e:surface-amalgam}),
where the free group $\Gamma_0$, freely generated by $B$, is identified with $F_{2g-2}$.
Denote the $\rho_0$-image of the product of commutators by 
\[
    W_0=\rho_0([a_2,b_2]\cdots[a_g,b_g]).
\]
Find $T_0,R_0\in K$ so that $[T_0,R_0]\cdot W_0=1$, by finding pairs of elements
$T_0^{(j)}, R_0^{(j)}\in \SU(2)$ satisfying 
$[T_0^{(j)}, R_0^{(j)}]=(W_0^{(j)})^{-1}$ for $j=1,\dots,d$.
We can now extend $\rho_0:\Gamma_0\overto{} K\overto{} G$ to
$\rho_{s,\textbf{r}}:\Gamma\overto{} G$ be setting
\[
    \begin{split}
        \rho_{s,\textbf{r}}(a_1)&=T_0 V_{s,\textbf{r}},\qquad
    \rho_{s,\textbf{r}}(b_1)=R_0,\\
    \rho_{s,\textbf{r}}(a_i)&=\rho_0(a_i),\qquad 
    \rho_{s,\textbf{r}}(b_i)=\rho_0(b_i),\qquad (i=2,\dots,d)
    \end{split}
\]
where the $j$th component of $V_{s,\textbf{r}}$ is a conjugate of the map
(\ref{e:j-loxo}) that lies in the centralizer of $R^{(j)}_0$.
here we use the fact that every element of $\SU(2)$ is a rotation fixing
a pair of points on the Riemann sphere; up to conjugation these points are $0$ and $\infty$
and then (\ref{e:j-loxo}) centralizes it.
Now properties (1)-(6) follow as in the free group case.
This completes the proof of Lemma~\ref{L:rho-s-r-properties}.

\end{proof}

\medskip

We need the following general Lemma. 

\begin{lemma}[Additivity of entropy]\label{L:additivity}
    Let $(M_1,\nu_1)$ and $(M_2,\nu_2)$ be two $(\Gamma,\mu)$-stationary spaces.
    Assume that the diagonal $\Gamma$-action on $M=M_1\times M_2$ has a $\mu$-stationary
    measure $\nu$ in measure class of the product $\nu_1\times\nu_2$ that projects to $\nu_1$ and $\nu_2$
    under $\pr_i:M\overto{} M_i$. Assume that
    $\log\frac{\dd\nu_1\otimes\nu_2}{\dd\nu}$ is in $L^1(M_1\times M_2,\nu)$.
    Then 
    \[
        h_\mu(M,\nu)=h_\mu(M_1,\nu_1)+h_\mu(M_2,\nu_2).
    \]
\end{lemma}
\begin{proof}
    For $g\in\Gamma$ and $x=(x_1,x_2)\in M$ we denote 
    \[
        \sigma_\nu(\gamma,x)=-\log\frac{\dd g_*^{-1}\nu}{\dd \nu}(x)
    \]
    the a.e. defined measurable cocycle $\sigma_\nu:\Gamma\times M\overto{} \bbR$. 
    Similarly we have cocycles $\sigma_{\nu_1\otimes\nu_2}:\Gamma\times M\overto{}\bbR$
    and $\sigma_{\nu_i}:\Gamma\times M_i\overto{} \bbR$. We note that
    \[
        \sigma_{\nu_1\otimes \nu_2}(g,(x_1,x_2))=\sigma_{\nu_1}(g,x_1)+\sigma_{\nu_2}(g,x_2).
    \]
    We now apply Lemma~\ref{L:cohomology} to deduce
    \[
    \begin{split}
        h_\mu(M,\nu)&=\int_\Gamma\int_M \sigma_\nu(g,x)\dd \nu(x)\dd\mu(g)\\
        &=\int_\Gamma\int_M \sigma_{\nu_1\otimes\nu_2}(g,(x_1,x_2))\dd \nu(x_1,x_2)\dd\mu(g)\\
        &=\int_\Gamma\int_M \left(\sigma_{\nu_1}(g,x_1)+\sigma_{\nu_2}(g,x_2)\right)\dd \nu(x_1,x_2)\dd\mu(g)\\
        &=\int_\Gamma\left(\int_M \sigma_{\nu_1}(g,x_1)\dd \nu(x_1,x_2)+
        \int_M \sigma_{\nu_2}(g,x_2)\dd \nu(x_1,x_2)\right)\dd\mu(g)\\
        &=\int_\Gamma\left(\int_{M_1} \sigma_{\nu_1}(g,x_1)\dd \nu_1(x_1)+
        \int_{M_2} \sigma_{\nu_2}(g,x_2)\dd \nu_2(x_2)\right)\dd\mu(g)\\
        &=h_\mu(M_1,\nu_1)+h_\mu(M_2,\nu_2).
    \end{split}
    \]
\end{proof}

We return to the proof of Theorem~\ref{T:main-cubes}. Let
\[
    h_d:=\Cent\cdot (\frac{\Csg\cdot \kappa_d}{d})^2
\]
For any $\textbf{r}=(r_1,\dots,r_d)$ with $r_1+\dots+r_d<\Csg\cdot \kappa_d$,
in particular for $r_j\in (0,\Csg\cdot \kappa_d/d)$,
for all $s\in (0,1]$ each one of the representations 
\begin{equation}\label{e:sr-rep}
    \rho_{s,\textbf{r}}:\Gamma\overto{} G=\PSL_2(\bbC)^d
\end{equation}
is in $\HomZdsg(\Gamma,G;\mu,A,B)$ (see Lemma~\ref{L:rho-s-r-properties}.(4)).
Therefore Theorem~\ref{T:Lebesgue-stat} applies, and
the $\Gamma$-action on $S=\prod_1^d S_j$, given by 
$\gamma: \xi\mapsto \rho_{s,\textbf{r}}(\gamma)\xi$,
has a unique $\mu$-stationary measure $\nu_{s,\textbf{r}}$; furthermore
this measure is equivalent to 
the $K$-invariant measure $m$ on $S$ with integrable $\log$-Radon-Nikodym derivatives.
This defines a $\mu$-boundary, and can be viewed as a point in $\Bnd(\Gamma,\mu)$.

Denote by $\nu^{(j)}_{s,\textbf{r}}=(\pr_j)_* \nu_{s,\textbf{r}}$ the projection 
to the $j$-th factor, which is a Riemann sphere $S_j$. It is a $\mu$-stationary measure
for the $\Gamma$-action via the $j$-th component representation
\[
    \Gamma\ \overto{\rho_{s,\textbf{r}}}\ G\ \overto{\pr_j}\ G_j=\PSL_2(\bbC)
\]
and defines a quotient of $\nu_{s,\textbf{r}}$ in $\Bnd(\Gamma,\mu)$.
The entropy $h_\mu(S_j,\nu^{(j)}_{s,\textbf{r}})$ of this projection varies continuously
in $r_j$ (Theorem~\ref{T:continuity}) and, using Theorem~\ref{T:lower_bound}
and the Intermediate Value Theorem, we conclude that any value $t_j\in (0,h_d)$
can be achieved by varying $r_j$.

We deduce that for every $s\in (0,1]$, for
any $d$-tuple $\textbf{t}=(t_1,\dots,t_d)$ in $(0,h_d)^d$ there exists 
$\textbf{r}=(r_1,\dots,r_d)$ so that
the representation $\rho_{s,\textbf{r}}$ defines a $\mu$-boundary
$(S,\nu_{s,\textbf{r}})$ with coordinate projections being $\mu$-boundaries with
\[
    h_\mu(S_j,\nu^{(j)}_{s,\textbf{r}})=t_j\qquad (j=1,\dots,d).
\]
Let us use notation $C_{s,\textbf{t}}\in\Bnd(\Gamma,\mu)$ for this $\mu$-boundary $(S,\nu_{s,\textbf{r}})$.
We claim that these $C_{s,\textbf{t}}$ have the desired properties.
Indeed, using repeated applications of Lemma~\ref{L:additivity}, we deduce
\[
    h_\mu(C_{s,\textbf{t}})=t_1+t_2+\dots+t_d.
\]
Let us fix the parameters $(s,\textbf{t})$, and denote the $\mu$-stationary
measure $\nu_{s,\textbf{r}}$ by $\nu$ on $S=\prod_{j=1}^d S_j$.

Let $J$ be a subset of $\{1,\dots,d\}$, denote by $S_J=\prod_{j\in J}S_j$ and 
$\pr_J:S\to S_J$ the corresponding quotient, and let $\nu_J=(\pr_J)_*\nu$ 
on $S_J$ be the corresponding quotient. Such a measure is absolutely continuous
with respect to $m_J=(\pr_J)_* m=\prod_{j\in J} m_j$ where $m_j$ are the $K_j$-invariant
probability measures on the Riemann spheres $S_j$ over $j\in J$.
Applying Lemma~\ref{L:additivity} we deduce that
\[
    h_\mu(S_J,\nu_J)=\sum_{j\in J} h_\mu(S_j,\nu_j)=\sum_{j\in J}t_j.
\]
Denoting $\omega\in \{0,1\}^d$ the indicator function of $J$ and by $C^\omega_{s,\textbf{t}}$
the quotient $(S_J,\nu_J)$ of $C_{s,\textbf{t}}=(S,\nu)$, we can rewrite the above
formula in the form
\[
    h_\mu(C^\omega_{s,\textbf{t}})=\sum_{j=1}^d \omega_j\cdot t_j
\]
in line with the statement of Theorem~\ref{T:main-cubes}.
Lemma~\ref{L:rho-s-r-properties}.(6) allows us to apply Proposition~\ref{P:qutients}
and deduce that measurable $\Gamma$-equivariant quotients of $C_{s,\textbf{t}}$
are $G$-equivariant. This shows that the collection 
$\{C^\omega_{s,\textbf{t}}\}$, $\omega\in \{0,1\}^d$,
is the \textit{complete} list of measurable $\Gamma$-quotients of $C_{s,\textbf{t}}$.
This collection forms a $d$-cube $\{0,1\}^d$ with 
$C_{s,\textbf{t}}=C^{(1,\dots,1)}_{s,\textbf{t}}$ and $C^{(0,\dots,0)}_{s,\textbf{t}}=\{*\}$
the trivial boundary.

Finally, Proposition~\ref{P:isoms} combined with Lemma~\ref{L:rho-s-r-properties}.(6)
guarantees that all the non-trivial quotients 
\[
    \{C^\omega_{s,\textbf{t}}\},\qquad (s\in (0,1],\ \omega\in \{0,1\}^d\setminus (0,\dots,0))
\]
are distinct. Moreover, in all cases the $\Gamma$-action is essentially free. 

This completes the proof of Theorem~\ref{T:main-cubes}, assuming Theorem~\ref{T:specgap-prods}
that we discuss in the next section.

\section{Spectral Gap for Products and Disjointness}\label{sec:disjointness}

Let $C$ be a finite set, $K$ a compact groups with Haar measure $m$,
and let $\tau:C\to K$ be a map. 
In this section, a \textit{spectral gap} $\kappa_\tau$ for such a map $\tau:C\to K$
refers to 
\[
    \kappa_\tau=1-\sup \setdef{\|Pf\|}{f\in L^2_0(K,m),\ \|f\|=1}
\]
for the self-adjoint Markov operator
\[
    Pf(x)=\frac{1}{2|C|}\left(\sum_{c\in C} f(\tau(c).x)+\sum_{c\in C}^kf(\tau(c)^{-1}.x)\right).
\]
This section is devoted to addressing the following problem.
\begin{problem}[Spectral Gap Disjointness]\label{G:disjoitness}
    Given a compact group $K$, a finite set $C$ and $n\in\bbN$, 
    construct maps $\tau_i:C\to K$, $1\le i\le n$, so that the diagonal embedding
    \[
        \tau_1\times\cdots\times \tau_n: C\ \overto{}\ K\times \cdots \times K
    \]
    has a spectral gap in the product $K^n$. Give a lower bound on this spectral gap
    in terms of $n$ and gaps $\kappa_{\tau_1},\dots,\kappa_{\tau_n}$.
\end{problem}
Fix $n\in\bbN$. Let $\Delta:K\to K^n$ denote the diagonal embedding, $\Delta(k)=(k,\dots,k)$. For an $n$-tuple $\bar{a}=(a_1,\dots,a_n)$ in $K^n$ we denote the embedding
\[
    \Delta^{\bar{a}}(k)=(k^{a_1},\dots,k^{a_n}),
    \qquad \textrm{where}\qquad
    k^a=a k a^{-1}.
\]
Let $B$ be a finite subset of $K$ with a spectral gap $\kappa_B>0$, 
and $\bar{a}=(a_1,\dots,a_n)$ be a $n$-tuple in $K^n$.
Double the finite set $B$ by taking
\[
    C=B\times\{1,2\}
\]
and consider the embeddings $\tau_i:C\overto{} K$ given by
\begin{equation}\label{e:tau-i}
        \tau_i(b,1)=b,\qquad \tau_i(b,2)=b^{a_i}\qquad 
        (1\le i\le n,\ b\in B).
\end{equation}
The Markov operator $P$ on $L^2_0(K^n,m^n)$ corresponding to this embedding
\[
    \bar\tau=\tau_1\times\cdots\times\tau_n:C=B\times \{1,2\}\ 
    \overto{}\ K^n
\]
is the average $P=\frac{1}{2}(P_1+P_2)$ of two self-adjoint Markov operators 
$P_1$ and $P_2$, corresponding to the maps
\[
    \Delta:B\ \overto{} \ K^n,\qquad \Delta^{\bar{a}}:B\ \overto{}\ K^n.
\]
Our approach applies to a general simple compact Lie group $K$, but we shall focus
on $K=\SO(3)$ (or rather its double cover ${\rm Spin}(3)=\SU(2)=\HHH$), 
because this is sufficient 
for our purposes, while some technical arguments in the proof are more 
explicit in this case. 

\begin{theorem}\label{T:spec-gap-disj}
    Let $B$ be a finite subset in $K=\SO(3)$ with a spectral gap $\kappa_{B}>0$.
    For $n\in\bbN$ there exists a $n$-tuple $\bar{a}=(a_1,\dots,a_n)$ in $K^n$ 
    so that taking $C=B\times\{1,2\}$ and maps $\tau_i:C\to K^n$, $1\le i\le n$, 
    as in (\ref{e:tau-i}) the Markov operator $P$ corresponding to the diagonal map
    $\bar\tau=\tau_1\times\cdots\times\tau_d:C\to K^n$ has a spectral gap of size
    \[
        \kappa_{\bar\tau}\ge c\cdot \frac{1}{n}\cdot \kappa_B>0
    \]
    where $c>0$ is a constant  associated with $\SO(3)$.
\end{theorem}

\subsection{General approach}
Consider the action of $K^n$ on itself by left translations, and
the corresponding unitary representation of $K^n$ on $L^2(K^n,m^n)$ 
and on the orthogonal complement to constants: 
\[
    \calH=L^2_0(K^n,m^n).
\]
We have a finite subset $B\subset K$ with a spectral gap $\kappa_B>0$
on $L^2_0(K,m)$, and some element $\bar{a}=(a_1,\dots,a_n)$ in $K^n$ 
that will be chosen later to ensure that $a_1,\dots,a_n$ are well separated in $K$. 
Associated with these choices there are two self-adjoint Markov operators: 
$P_1$ and $P_2$ on $K^n$, defined by the subsets
\[
    \Delta(B),\qquad \Delta^{\bar{a}}(B)
\]
in $K^n$. Since the constants are preserved by these operators, 
$P_1$ and $P_2$ act on $\calH$.

Let $\calH_1, \calH_2$ denote the subspaces of invariants for subgroups
$\Delta(K)$ and $\Delta^{\bar{a}}(K)$, respectfully.
More precisely:
\begin{equation}\label{e:calH12}
    \begin{split}
        \calH_1&:=\setdef{f\in \calH}{\forall k\in K,\ 
            f(k x_1,\dots,k x_n)=f(x_1,\dots,x_n)},\\
        \calH_2&:=\setdef{f\in \calH}{\forall k\in K,\ 
            f(k^{a_1}x_1,\dots,k^{a_n}x_n)=f(x_1,\dots,x_n)}.
    \end{split}
\end{equation}
\begin{lemma}\label{L:Hiperp-contraction}
    For $i=1,2$ the splitting $\calH=\calH_i\oplus \calH_i^\perp$ 
    is $P_i$-invariant, where $P_i$ is the identity on $\calH_i$,
    and $P_i$ is a $(1-\kappa_B)$-contraction on $\calH_i^\perp$.  
\end{lemma}
\begin{proof}
    We focus on the case $i=1$, as $i=2$ follows by applying 
    the conjugation by $\bar{a}\in K^n$.
    Since $\Delta(B)\subset\Delta(K)$, 
    the subspace $\calH_1$ is $P_1$ invariant
    and $P_1$ is the identity on $\calH_1$.
    Being self-adjoint, $P_1$ preserves $\calH_1^\perp$.
    
    The projection $K^n\to \Delta(K)\backslash K^n$ has fibers that 
    can be identified with $K\cong \Delta(K)$, and the orthogonal
    projection $\calH\to\calH_1$, $f\mapsto \bar{f}$, given by
    \[
        \bar{f}(x_1,\dots,x_n)=\int_K f(kx_1,\dots,kx_n)\dd m(k),
    \]
    is by integration over these fibers.
    Thus $f\in \calH_1^\perp$ iff the disintegration 
    \[
        f=\int_{\Delta(K)\backslash K^n} f_y
    \]
    of $f$ along these fibers has $f_y\in L^2_0(K,m)$ a.e.  
    The Markov operator $P_1$ preserves each one of these fibers,
    acting as $(1-\kappa_B)$-contraction on each $f_y$. Thus
    \[
        f\in \calH_1^\perp\qquad\Longrightarrow\qquad
        \|P_1f\|\le (1-\kappa_B)\cdot \|f\|
    \]
    which is precisely the claim that $P_1$ is $(1-\kappa_B)$-contraction on $\calH_1^\perp$.
\end{proof}

\medskip

Define the \textit{angular separation} of the subspaces $\calH_1$ and $\calH_2$ by
\begin{equation}\label{e:delta}
    \delta(\calH_1,\calH_2):=1-\sup\setdef{\frac{|\ip{f_1}{f_2}|}{\|f_1\|\cdot\|f_2\|}}{f_i\in \calH_i\setminus\{0\}}.
\end{equation}
(Note that $\calH_1\cap\calH_2=\{0\}$ is necessary but not sufficient 
for positive separation in infinite dimensional Hilbert spaces). 
% We will show that positive angular separation $\delta(\calH_1,\calH_2)>0$ leads
% an estimate for a spectral gap for $P=\half(P_1+P_2)$ as follows. 
\begin{lemma}
    Suppose $P_1$, $P_2$ are self-adjoint Markov operators on $\calH$, 
    where $P_i$ fixes point-wise subspace $\calH_i$ vectors, and acts as
    a $(1-\kappa)$-contraction on $\calH_i^\perp$, and assume that 
    $\calH_1,\calH_2$ have positive angular separation $\delta=\delta(\calH_1,\calH_2)>0$.
    Then the average $P=\half(P_1+P_2)$ is a $(1-\kappa\cdot\delta^2/36)$ contraction on $\calH$.
\end{lemma}
\begin{proof}
    Take arbitrary $f\in \calH$ with $\|f\|=1$. Since $\|Pf\|\le \half \|P_1f\|+\half\|P_2f\|$,
    while $\|P_i f\|\le 1$, it suffices to show that
    \[
        \min_{i=1,2} \| P_if\|\le 1-\frac{\kappa\cdot \delta^2}{18}.
    \]
    Write $f=f_1+g_1=f_2+g_2$ with $f_i\in\calH_i$ and $g_i\in\calH_i^\perp$. 
    Then $1=\|f\|^2=\|f_i\|^2+\|g_i\|^2$ for $i=1,2$.
    We first note that $f$ cannot be very close to both $\calH_1$ and $\calH_2$.
    More precisely, the claim is that
    \[
        \max(\|g_1\|,\|g_2\|)> \frac{\delta}{3}.
    \]
    Indeed, otherwise $\|g_i\|\le \delta/3$ and $\|f_i\|\ge \sqrt{1-\delta^2/9}$,
    and we get
    \[
    \begin{split}
        1=\ip{f}{f}&=\ip{f_1+g_1}{f_2+g_2}\\
        &\le|\ip{f_1}{f_2}|+|\ip{f_1}{g_2}|+|\ip{g_1}{f_2}|+|\ip{g_1}{g_2}|\\
        &\le (1-\delta)\cdot\|f_1\|\cdot\|f_2\|+\|g_1\|+\|g_2\|+\|g_1\|\cdot\|g_2\|\\
        &\le (1-\delta)\cdot (1-\frac{\delta^2}{9})+\frac{\delta}{3}+\frac{\delta}{3}
        +\frac{\delta^2}{9}
        \le 1-\frac{\delta}{3}+\frac{\delta^3}{9}<1
    \end{split}
    \]
    which is absurd. 
    
    Assuming $\|g_1\|\ge \delta/3$ we will show $\|P_1f\|\le 1-\kappa\delta^2/18$.
    The argument for the other case is identical. 
    Since $\|f_1\|=\sqrt{1-\|g_1\|^2}$, $P_1f_1=f_1$, $P_1g_1\perp f_1$, 
    and $(1-\kappa)^2\le 1-\kappa$, we have
    \[
        \begin{split}
        \|P_1 f\|^2 &=\|f_1\|^2+\|P_1 g_1\|^2 
        \le (1-\|g_1\|^2) +(1-\kappa)^2\cdot\|g_1\|^2\\
        &\le 1-\|g_1\|^2+(1-\kappa)\cdot\|g_1\|^2=1-\kappa\cdot \|g_1\|^2\\
        &\le 1-\frac{\kappa\cdot \delta^2}{9}.
        \end{split}
    \]
    For $t\in (0,1)$ we have $\sqrt{1-t}< 1-t/2$. Hence we get
    $\|P_1f\|<1-\kappa\cdot \delta^2/18$. 
    As was pointed out at the beginning of the proof this estimate gives  
    \[
        \|Pf\|\le 1-\frac{\kappa\cdot \delta^2}{36}
    \]
    as required.
\end{proof}
The next step for the proof of Theorem~\ref{T:spec-gap-disj} is to choose
$\bar{a}=(a_1,\dots,a_n)\in K^n$ to ensure that the subspaces $\calH_1$, $\calH_2$
have a positive angular separation, more precisely
\[
    \delta(\calH_1,\calH_2)\ge \frac{c}{\sqrt{n}}\ ??
\]
for some constant $c>0$.
To this end we identify $\calH_i$ with $L^2_0(K^{n-1},m^{n-1})$ as follows:
the space $\calH_1$ consists of functions of the form
\[
    f(x_1,\dots,x_n)=\phi(x_1^{-1}\cdot x_2,\dots,x_1^{-1}\cdot x_{x_{n}})
\]
where $\phi\in L^2_0(K^{n-1},m^{n-1})$. Similarly, a general  $g\in \calH_2$
is obtained from $\psi\in L^2_0(K^{n-1},m^{n-1})$ be the formula
\[
\begin{split}
    g(x_1,\dots,x_n)&= \psi((x_1^{a_1})^{-1}\cdot x_2^{a_2},\dots,(x_1^{a_1})^{-1}\cdot x_n^{a_n})\\
    &=\psi(a_1x_1^{-1}(a_1^{-1}a_2)x_2a_2^{-1},\dots,a_1x_1^{-1}(a_1^{-1}a_n)x_n a_n^{-1})\\
    &=\psi'(x_1^{-1}c_2x_2,\dots,x_1^{-1}c_n x_n),
\end{split}
\]
where $c_i=a_1\iv a_i$ for $i=2,\dots,n$ and $\psi'$ is $\psi$ translated from the left 
by $(a_1,\dots,a_1)$ and from the right by $(a_2^{-1},\dots,a_n\iv)$.
Since the maps
\[
    L^2_0(K^{n-1},m^{n-1})\overto{} \calH_1,\ \phi\mapsto f,
    \qquad 
    L^2_0(K^{n-1},m^{n-1})\overto{} \calH_2,\ \psi\mapsto g
\]
are unitary isomorphisms, the inner product of unit vectors 
$f\in \calH_1$ and $g\in\calH_2$ is given by
\[
\begin{split}
    \ip{f}{g}&=\int_{K^n} \phi(x_1^{-1}x_2,\dots,x_1^{-1}x_n)\cdot
    \bar\psi'(x_1^{-1}b_2x_2,\dots,x_1^{-1}b_nx_n)\dd x_1\dots \dd x_n\\
    &=\int_{K^{n-1}} \phi(y_2,\dots,y_n)\cdot \left(\int_K \bar\psi'(c_2^x\cdot y_2,\dots,c_n^x\cdot y_n)\dd x\right)
    \dd y_2\dots \dd y_n\\
    &= \ip{\phi}{ Q_\mu \psi'}
\end{split}
\]
where $\phi,\psi,\psi'$ are unit vectors in $L^2_0(K^{n-1},m^{n-1})$, and $Q_\mu$ is the 
Markov operator on $K^{n-1}$ given by the convolution with the probability measure
that describes the distribution of $(c_2^x,\dots,c_n^x)$ when $x\in K$ is uniformly distributed:
\[
    \mu=\int_K \delta_{(c_2^x,\dots,c_n^x)}\dd m(x).
\]
Hence the for all $f\in \calH_1$, $g\in\calH_2$ with $\|f\|=\|g\|=1$ we have the bound
\[
    |\ip{f}{g}|\le |\ip{\phi}{Q_\mu \psi'}|\le \|Q_\mu\|
\]
by the operator norm of $Q_\mu$ on $L^2_0(K^{n-1},m^{n-1})$.
Thus the angular separation between $\calH_1$ and $\calH_2$ in $L^2_0(K^n,m^n)$ 
is bounded from below by this spectral gap:
\begin{equation}
    \delta(\calH_1,\calH_2)\ge 1-\|Q_\mu\|.
\end{equation}
We shall now specialize to $K=\SO(3)$ to estimate the spectral gap $1-\|Q_\mu\|$.

\subsection{A spectral gap of $Q_\mu$}
We consider $\mu$ and its convolutions $\mu*\mu$ etc. on $H=K^n$ 
where $K=\SO(3)$. 
These are probability measures on $K^n$ invariant under conjugation by $L=\Delta(K)$.

More generally, let $L$ be  be a closed subgroup of a compact metrizable group $H$, 
and consider its action $\ell: h\mapsto h^\ell=\ell h\ell\iv$ by conjugation on $H$.
We denote by 
\[
    \Prob(H)^L=\setdef{\mu\in\Prob(H)}{\forall\ell\in L:\ \mu^\ell=\mu}
\]
the space of probability measures on $H$ invariant 
under conjugation by elements $\ell\in L$, and by $H^L$
the space of $L$-conjugacy classes $\setdef{h^\ell}{\ell\in L}$ in $H$.
For $x\in H^L$ denote by
\[
    \mu_x=\int_L \delta_{h^\ell}\dd m_L(\ell)
\]
the $L$-uniform distribution on the $L$-conjugacy class $x$ represented by some/any $h\in x$.
\begin{lemma}\label{L:probHL}
    Let $L<H$ be a closed subgroup of a compact group $H$. Then
    \begin{enumerate}
        \item The space $H^L$ is a compact space.
        \item The closed convex subspace $\Prob(H)^L\subset\Prob(H)$ 
        is closed under convolution.
        \item There is a natural identification $\Prob(H^L)\cong\Prob(H)^L$.
        In particular, $\Prob(H)^L$ is a convex weak-* compact with 
        extreme points $\setdef{\mu_x}{x\in H^L}$.
        \item For any $x,y\in H^L$ there is a unique 
        measure $\omega_{x,y}\in\Prob(H^L)$ so that
        \[
            \mu_x*\mu_y=\int_{H^L} \mu_z\dd\omega_{x,y}(z).
        \]
    \end{enumerate}
\end{lemma}
\begin{proof}
    (1) $L$ is a compact group being a closed subgroup in a compact group $H$.
    It acts on $H$ by conjugation, hence the space $H^L$ of the orbits of this action 
    is Hausdorff, and since $H$ is compact while $H\to H^L$ is continuous, $H^L$ is a compact space.
    
    (2) For $\mu,\eta\in \Prob(H)^L$ and $\ell\in L$ we have
    \[
        (\mu*\eta)^\ell=\int_H\int_H\delta_{(h_1h_2)^\ell}\dd\mu(h_1)\dd\eta(h_2)
        =\int_H\int_H \delta_{h_1' h_2'}\dd\mu^\ell(h_1')\dd\eta^\ell(h_2')=\mu*\eta
    \]
    because $(h_1h_2)^\ell=h_1^\ell h_2^\ell$.

    (3) The identification $\Prob(H^L)\cong\Prob(H)^L$ follows from the fact that $H^L$
    is Hausdorff and that any $L$-orbit $x=h^L$ has the form $L/L_0$ for a closed subgroup 
    $L_0=C_L(x)$ and therefore carries a unique $L$-invariant probability measure $\mu_x$.
    By Riesz representation theorem every $\eta\in \Prob(H)^L=\Prob(H^L)$ 
    decomposes uniquely as an integral over its extreme points $x\in H^L$, i.e. as
    an integral of $\mu_x$.
    
    (4) By (2) $\mu_x*\mu_y\in \Prob(H^L)$ and by (3) it decomposes uniquely as
    an integral of $\mu_z$, $z\in H^L$.
\end{proof}

\subsection*{The case of $L=H=\SO(3)$}
The group $\SO(3)$ acts on the unit sphere $S^2$ by isometries.
Every $k\in \SO(3)$ is a rotation around some axis, 
and the angle $\Phi(k)\in [0,\pi]$ of this rotation parametrizes 
the conjugacy class of $k$. In fact, we have a homeomorphism
\[
    \SO(3)^{\SO(3)}\ \overto{\Phi}\ [0,\pi].
\]
For $\phi\in [0,\pi]$ the measure $\mu_\phi$ on $\Phi^{-1}(\{\phi\})$ 
corresponds to rotations by angle $\phi$ about an axis that is chosen uniformly 
at random over the $2$-sphere $S^2$.

\medskip

We would like to compute the measure $\omega_{\alpha,\beta}$ for $\mu_\alpha*\mu_\beta$
(as in Lemma~\ref{L:probHL}),
i.e. the distribution of the rotation angle $\Phi(ab)$ of a composition of a random 
$\alpha$-rotation $a$
with  a random $\beta$-rotation $b$. 
It is more convenient to do so for the double cover ${\rm Spin}(3)$ of $\SO(3)$
that can be identified with $\SU(2)$ and with the unit quaternions $\HHH$.
We take the latter perspective.

\subsection*{The case of $L=H=\HHH$}

The group of unit quaternions is
\[
    \HHH=\setdef{a+b\mathbf{i}+c\mathbf{j}+d\mathbf{k}}{a,b,c,d\in\bbR,
    \ a^2+b^2+c^2+d^2=1}
\]
with $\mathbf{i}^2=\mathbf{j}^2=\mathbf{k}^2=-1$ and $\mathbf{i}\mathbf{j}\mathbf{k}=1$.
Any quaternion can be written as $p=r+\mathbf{v}$ where $r$ is the \textit{scalar part} 
and the \textit{vector part} $\mathbf{v}\in {\rm Span}(\mathbf{i},\mathbf{j},\mathbf{k})$.
Hence a unit quaternion $p$ can expressed as 
\[
    p=\cos(\Theta)+\sin(\Theta)\cdot \mathbf{u}
\]
for a unique $\Theta\in [0,\pi]$
and a unit vector $\mathbf{u}$ ($\mathbf{u}$ is uniquely defined when $\Theta\ne 0,\pi$).
The image $k\in\SO(3)$ of such $p\in \HHH$ is given by a rotation 
by angle $2\cdot\Theta$ around the axis $\mathbf{u}$.
Two unit quaternions are conjugate iff they have the same angular part, so
\[
    \HHH^{\HHH}\ \overto{\Theta}\ [0,\pi]
\]
is a parametrization of $\HHH$ conjugacy classes. 
For $\phi\in [0,\pi]$ the measure $\mu_\phi\in \Prob(\HHH)^\HHH$ is
\[
    \mu_\phi=\int_{S^2} \delta_{\cos(\phi)+\sin(\phi)\cdot \mathbf{u}} \dd \mathbf{u}.
\]
For an arbitrary angle $\phi$ denote $\bar\phi=\arccos(\cos(\phi))$;
in other words $\bar\phi$ is the unique value in $[0,\pi]$ with $\cos(\phi)=\cos(\bar\phi)$.
\begin{lemma}[Implicit in Archimedes !]\label{L:Archimedes}\hfill{}\\
    For $\alpha,\beta\in [0,\pi]$ we have
    \[
        \mu_\alpha*\mu_\beta=\frac{1}{2\sin \alpha\sin\beta}
        \cdot
        \int_{\overline{\alpha-\beta}}^{\overline{\alpha+\beta}}\mu_{\phi}\cdot\sin(\phi) \dd \phi.
    \]
    In particular, for $\alpha\in (0,\pi/2]$ 
    \[
        \mu_\alpha*\mu_\alpha=\frac{1}{2\sin^2\alpha}\cdot\int_{0}^{2\alpha} \mu_\phi\cdot \sin(\phi)\dd\phi.
    \]
\end{lemma}
\begin{proof}
    Multiplication of quaternions can be expressed as
    \[
        (r_1+\mathbf{v}_1)\cdot (r_2+\mathbf{v}_2)
        =\left(r_1r_2-\ip{\mathbf{v}_1}{\mathbf{v}_2}\right)
        +\left(r_1\mathbf{v}_2+r_2 \mathbf{v}_1+\mathbf{v}_1\times \mathbf{v}_2\right).
    \]
    In particular, the scalar part of the products $p\cdot q$ of two unit quaternions
    \[
        p=\cos\alpha +\sin \alpha\cdot \mathbf{u},
        \qquad q=\cos\beta +\sin \beta\cdot \mathbf{w}
    \]
    with unit vectors $\mathbf{u}$, $\mathbf{w}$, is given by
    \begin{equation}\label{e:scalarcomp}
      \cos\alpha\cdot \cos\beta-\sin\alpha\cdot \sin\beta \cdot\ip{\mathbf{u}}{\mathbf{w}}.
    \end{equation}
    The distribution of $\ip{\mathbf{u}}{\mathbf{w}}$ for $\mathbf{u}$ and $\mathbf{w}$
    being chosen independently and uniformly over $S^2$, is the same as the distribution
    of a projection of a single random $\mathbf{u}\in S^2$ to a fixed axis. 
    In his famous computation of surface area for $S^2$ Archimedes (!) 
    showed that this projection is uniformly distributed over $[-1,1]$.
   
    The extreme cases of $\ip{\mathbf{u}}{\mathbf{w}}=\pm1$ correspond to
    the value of $\cos(\alpha\pm\beta)$ for (\ref{e:scalarcomp}).
    Therefore the scalar part of $p\cdot q$ is a uniform distribution
    over the interval with endpoints 
    \[
        \cos(\alpha-\beta)=\cos(\overline{\alpha-\beta}),
        \qquad 
        \cos(\alpha+\beta)=\cos(\overline{\alpha+\beta}),
    \]
    while $\Theta(p\cdot q)$ distributes over the interval with endpoints
    $\overline{\alpha-\beta}$ and $\overline{\alpha+\beta}$
    with density $\dd \cos\phi=-\sin(\phi)\dd\phi$, appropriately rescaled.
    The rescaling is by 
    \[
        \cos(\overline{\alpha-\beta})-\cos(\overline{\alpha+\beta})
        =\cos(\alpha-\beta)-\cos(\alpha+\beta)=2\sin\alpha\sin\beta.
    \]
\end{proof}

\medskip

In $\SO(3)$ the angle of rotation $\Phi(k)$ can be expressed as the maximal displacement 
\[
    \Phi(k)=\max_{\xi\in S^2} d(k\xi,\xi)
\]
of points on the unit 2-sphere. 
This implies that $\Phi(k_1^{-1}k_2)$ is a bi-invariant metric on $\SO(3)$.
It follows that $\Phi(p^{-1}q)$ is a bi-invariant metric on $\HHH$.
For $\alpha>0$ we denote by
\[
    U_\alpha=\setdef{p\in\HHH}{\Phi(p)<\alpha}
\]
the open $\alpha$-neighborhood of the identity in $\HHH$, and by
\[
    m_\alpha=\frac{1}{m(U_\alpha)}\cdot m|_{U_\alpha}
\]
the normalized restriction of the Haar measure to this neighborhood.
\begin{lemma}
    The Haar measure $m$ on $\HHH$ and $m_\alpha$ are both conjugation invariant and given by
    \[
        m=\int_0^\pi \frac{2\sin^2\phi}{\pi}\cdot  \mu_\phi\dd\phi,
        \qquad 
        m_\alpha=\frac{1}{m(U_\alpha)}\cdot \int_0^\alpha 
         \frac{2\sin^2\phi}{\pi}\cdot\mu_\phi\dd\phi,
    \]
    where 
    \[
        m(U_\alpha)=\frac{\alpha-\sin(\alpha)\cdot\cos(\alpha)}{\pi}
        = \frac{4}{3\pi}\cdot \alpha^3+O(\alpha^5).
    \]
    In particular, $\mu_\alpha^{*2}=\mu_\alpha*\mu_\alpha$ and $m_{2\alpha}$ 
    are in the same measure class, and  
    the Radon-Nikodym derivative is
    \[
        \frac{\dd(\mu_\alpha^{*2})}{\dd (m_{2\alpha})}(\phi)
        =\left(\frac{8\alpha}{3}+O(\alpha^3)\right)\cdot \frac{1}{\sin\phi}
        \cdot 1_{(0,2\alpha)}(\phi).
    \]
\end{lemma}
\begin{proof}
    Unit quaternions $\HHH$ can be identified with the $3$-sphere $S^3$,
    and the Haar measure is the normalized sphere-volume.
\end{proof}

\medskip

\begin{lemma}
    For any $n\in\bbN$ there exist $p_1,\dots,p_n\in\HHH$, 
    so that 
    \[
        \Theta(p_i^{-1}p_j)\ge \frac{1}{100\cdot n^{\frac{1}{3}}} 
    \]
    for any $i\ne j$ in $\{1,\dots,n\}$.
\end{lemma}
This estimate can be obtained by choosing a maximally separated $n$-net in $\HHH$;
the power $1/3$ is related to the fact that $\HHH$ is $3$-dimensional;
the constant $1/100$ is just some crude lower bound on the relevant geometry.

\end{document}